\documentclass[a4paper,10pt]{article}

\newif \ifDdP \DdPtrue

\usepackage{amsmath,amsfonts,amsthm}
\usepackage{a4wide}
\usepackage{authblk}
\usepackage{booktabs}
\usepackage{cite}
\usepackage{comment}
\usepackage{enumerate}
\usepackage{general}
\usepackage{hyperref}
\usepackage[bold]{hhtensor}
\usepackage[latin1]{inputenc}
\usepackage{nicefrac}
\usepackage{paralist}

\ifDdP
\usepackage{pgfplots,pgfplotstable}
\usepackage{subcaption}
\fi

\usepackage{caption}
\usepackage{tabularx}
\usepackage{xcolor}

\usetikzlibrary{external}
\newcommand{\email}[1]{\href{mailto:#1}{#1}}

\ifDdP
\graphicspath{{./fig/pdf/}}
\fi

\newcommand{\vel}{\vec{\beta}}
\newcommand{\diff}[1][]{\nu_{#1}}
\newcommand{\ldiff}{\underline{\nu}}
\newcommand{\reac}{\mu}
\newcommand{\Pe}{\mathrm{Pe}}
\newcommand{\Da}{\mathrm{Da}}

\newcommand{\flux}[1][]{\vec{\Phi}_{#1}}

\newcommand{\Gam}{\Gamma_{\diff,\vel}}
\newcommand{\Intnu}[1][]{{\cal I}_{\diff}^{#1}}
\newcommand{\Int}[1][]{{\cal I}_{\diff,\vel}^{#1}}

\newcommand{\Tdiff}[1][F]{T_{\diff}(#1)}

\newcommand{\UT}{\underline{\mathsf{U}}_T^k}
\newcommand{\Uh}[1][]{\underline{\mathsf{U}}_{h#1}^k}
\newcommand{\Uhz}[1][]{\underline{\mathsf{U}}_{h#1}^0}

\newcommand{\su}[1][T]{\underline{\mathsf{u}}_{#1}}
\newcommand{\sv}[1][T]{\underline{\mathsf{v}}_{#1}}
\newcommand{\sw}[1][T]{\underline{\mathsf{w}}_{#1}}
\newcommand{\sz}[1][T]{\underline{\mathsf{z}}_{#1}}

\newcommand{\unu}[1][]{\mathsf{u}_{#1}}
\newcommand{\unv}[1][]{\mathsf{v}_{#1}}
\newcommand{\unw}[1][]{\mathsf{w}_{#1}}
\newcommand{\unz}[1][]{\mathsf{z}_{#1}}

\newcommand{\shu}[1][]{\underline{\widehat{\mathsf{u}}}_{#1}}
\newcommand{\unhu}[1][T]{\widehat{\mathsf{u}}_{#1}}

\newcommand{\cu}[1][T]{\widecheck{u}_{#1}}

\newcommand{\pT}[1][k]{p_T^{#1}}
\newcommand{\pTF}[1][k]{p_{T(F)}^{#1}}
\newcommand{\PT}[1][k]{P_T^{#1}}
\newcommand{\GBT}[1][k]{G_{\vel,T}^{#1}}

\newcommand{\velc}[1][T]{\beta_{{\rm ref},#1}}
\newcommand{\tc}[1][T]{\tau_{{\rm ref},#1}}
\newcommand{\Lvel}[1][T]{L_{\vel,#1}}

\newcommand{\lproj}[2][h]{\pi_{#1}^{#2}}

\newcommand{\IT}[1][k]{\underline{\mathsf{I}}_T^{#1}}
\newcommand{\Ih}[1][k]{\underline{\mathsf{I}}_h^{#1}}

\newcommand{\ppos}{^{+}}
\newcommand{\pneg}{^{-}}

\newcommand{\jump}[1]{[\![#1]\!]}

\newcommand{\cf}{\mathbf{1}}

\newcommand{\Ap}{A^+}
\newcommand{\Am}{A^-}
\newcommand{\Apm}{A^{\pm}}
\newcommand{\Aabs}{|A|}
\newcommand{\ASSUM}[2]{(\textsf{#1#2})}

\ifDdP
\else
\newcommand{\bigtimes}{\times}
\newcommand{\widecheck}{\check}
\fi

\usepackage[normalem]{ulem}
\definecolor{violet}{rgb}{0.580,0.,0.827}

\title{A discontinuous-skeletal method for advection-diffusion-reaction on general meshes}
\author[1]{Daniele A. Di Pietro\footnote{\email{daniele.di-pietro@umontpellier.fr}, corresponding author}}
\affil[1]{University of Montpellier, Institut Montpelli\'{e}rain Alexander Grothendieck, 34095 Montpellier, France}
\author[2]{J\'{e}r\^{o}me Droniou\footnote{\email{jerome.droniou@monash.edu}}}
\affil[2]{School of Mathematical Sciences, Monash University, Victoria 3800, Australia}
\author[3]{Alexandre Ern\footnote{\email{ern@cermics.enpc.fr}}}
\affil[3]{University Paris-Est, CERMICS (ENPC), 
  6--8 avenue Blaise Pascal,
  77455 Marne-la-Vall\'ee CEDEX 2, France
}

\begin{document}

\maketitle

\begin{abstract}
We design and analyze an approximation method for advection-diffusion-reaction equations where the (generalized) degrees of freedom are polynomials of order $k\ge0$ at mesh faces. The method hinges on local discrete reconstruction operators for the diffusive and advective derivatives and a weak enforcement of boundary conditions. Fairly general meshes with polytopal and nonmatching cells are supported. Arbitrary polynomial orders can be considered, including the case $k=0$, which is closely related to Mimetic Finite Difference/Mixed-Hybrid Finite Volume methods. The error analysis covers the full range of P\'eclet numbers, including the delicate case of local degeneracy where diffusion vanishes on a strict subset of the domain. Computational costs remain moderate since the use of face unknowns leads to a compact stencil with reduced communications. Numerical results are presented.
\end{abstract}

\section{Introduction}

The goal of the present work is to design and analyze an approximation method for advection-diffusion-reaction equations where the (generalized) degrees of freedom (DOFs) are polynomials of order $k\ge0$ at mesh faces. Since such faces constitute the mesh skeleton, and since DOFs can be chosen independently at each face, we use the terminology discontinuous-skeletal method. The proposed method offers various assets: (i) Fairly general meshes, with polytopal and nonmatching cells, are supported; (ii) Arbitrary polynomial orders, including the case $k=0$, can be considered; (iii) The error analysis covers the full range of P\'eclet numbers; (iv) Computational costs remain moderate since skeletal DOFs lead to a compact stencil with reduced communications. 

Approximation methods using face-based DOFs have been investigated recently for advection-diffusion equations on meshes composed of standard elements.
In~\cite{Cockburn.Dong.ea:09}, Cockburn et al. devise and numerically investigate a Hybridizable Discontinuous Galerkin (HDG)  method for the diffusion-dominated regime based on a mixed formulation where an approximation for the total advective-diffusive flux is sought.
In~\cite{Chen.Cockburn:14}, Chen and Cockburn carry out a convergence analysis for a variable degree HDG method on semimatching nonconforming simplicial meshes, and investigate the impact of mesh nonconformity on the supercloseness of the potential. The formulation differs from~\cite{Cockburn.Dong.ea:09} in that the flux variable now approximates the diffusive component only.
In~\cite{Wang.Ye:13}, Wang and Ye analyze a Weak Galerkin method for advection-diffusion-reaction on triangular meshes, which appears to be mainly tailored to the diffusion-dominated case.
Turning to low-order methods on general polyhedral meshes, we cite, in particular, the work of Beir\~{a}o da Veiga, Droniou, and Manzini~\cite{Beirao-da-Veiga.Droniou.ea:10} on Hybrid Mimetic Mixed (HMM) methods (which encompass, see~\cite{Droniou.Eymard.ea:10}, three families of numerical schemes for elliptic equations:
the Mimetic Finite Difference method \cite{Brezzi.Lipnikov.ea:05}, the Mixed Finite Volume method
\cite{Droniou.Eymard:06}, and the Hybrid Finite Volume -- or SUSHI -- method
\cite{Eymard.Gallouet.ea:10}).
Although the analysis focuses on the diffusion-dominated case, we show here that a suitable tweaking of the scheme so as to include weakly enforced boundary conditions allows one to treat the advection-dominated case as well.%

The starting point for the present discontinuous-skeletal method is the Hybrid-High Order (HHO) method designed in~\cite{Di-Pietro.Ern.ea:14,Di-Pietro.Ern:14a} for purely diffusive and linear elasticity problems. The key ideas in~\cite{Di-Pietro.Ern.ea:14,Di-Pietro.Ern:14a} are as follows: (i) In each mesh cell, a local potential reconstruction of order $(k+1)$ is devised from polynomials of order $k$ in the cell and on its faces (cell- and face-based DOFs); (ii) A local bilinear form is built using a Galerkin form based on the gradient of the local potential reconstruction plus a stabilization form which preserves the improved order of the reconstruction; this leads to energy-error estimates of order $(k+1)$ and $L^2$-potential estimates of order $(k+2)$ if elliptic regularity holds; (iii) The global discrete problem is assembled cellwise, and cell-based DOFs are eliminated by static condensation, so that only the face-based DOFs remain. 

The extension to advection-diffusion-reaction equations entails several new ideas: (i) We devise a local reconstruction of the advective derivative from cell- and face-based DOFs using an integration by parts formula; (ii) Stability for the advective contribution is ensured by terms that penalize the difference between cell- and face-based DOFs at faces, and which therefore do not preclude the possibility of performing static condensation and do not enlarge the stencil; as in~\cite{Beirao-da-Veiga.Droniou.ea:10}, the stability terms are formulated in a rather general form so as to include various approaches used in the literature, e.g., upwind, locally $\theta$-upwind, and Scharfetter--Gummel schemes; (iii) Boundary conditions are enforced weakly so as to achieve robustness in the full range of P\'eclet numbers. 

An additional novel feature of the present work is that our analysis also includes the case of locally degenerate advection-diffusion-reaction equations, where the diffusion coefficient vanishes on a (strict) subset of the computational domain. We emphasize that such problems are particularly delicate since the exact solution can jump at the diffusive/nondiffusive interface separating zero and nonzero regions for the diffusion coefficient. The literature on locally degenerate advection-diffusion-reaction problems is relatively scarce.
The coupling of parabolic-hyperbolic systems in one space dimension is considered by Gastaldi and Quarteroni~\cite{Gastaldi.Quarteroni:89}; see also~\cite{Ern.Proft:06}.
In both cases, ad hoc techniques are proposed based on removing suitable terms at the diffusive/nondiffusive interface.
In~\cite{Di-Pietro.Ern.ea:08}, a discontinuous Galerkin (dG) method is designed and analyzed, where the use of weighted averages allows one to automatically handle the possibility of jumps in the exact solution. A dG method handling the degenerate case is also considered numerically by Houston, Schwab, and S\"{u}li~\cite{Houston.Schwab.ea:02}. In the present setting, a -- perhaps surprising -- result at first sight is that an approximation method hinging on face-based DOFs can capture well a discontinuous solution. This result is achieved owing to a tailored design of the stabilization terms ensuring that the interface unknown approximates well the exact solution from the diffusive side.
 
The material is organized as follows.
In Section~\ref{sec:cont}, we present the model problem.
In Section~\ref{sec:setting}, we introduce the discrete setting. 
In Section~\ref{sec:disc.op}, we define the local bilinear forms and introduce the novel ideas to discretize the advective terms.  
In Section~\ref{sec:method}, we build the discrete problem, introduce several norms for the analysis, and state our main results on stability and error estimates. The dependence on the physical parameters is tracked in the error estimate so as to capture the variation in convergence order between the diffusive and the advective regimes. We also study the link with the HMM methods of~\cite{Beirao-da-Veiga.Droniou.ea:10} in the case $k=0$. Such a link was already noticed in \cite{Di-Pietro.Ern.ea:14} for HHO methods in the purely diffusive case.
Numerical results on standard and general polygonal meshes are presented to assess the sharpness of the error estimate in both the uniformly vanishing diffusion and locally degenerate cases.
Finally, in Section~\ref{sec:analysis}, we prove our main results.  

\section{Model problem}
\label{sec:cont}

Let $\Omega\subset\Real^d$, $d\ge 1$, be an open bounded connected polytope of boundary $\partial\Omega$ and unit outer normal $\normal$.
We denote by $\diff:\Omega\to\Real^+$ the diffusion coefficient, which we assume to be piecewise constant on a partition $P_\Omega\eqbydef\{\Omega_i\}_{1\le i\le N_\Omega}$ of $\Omega$ into polytopes and such that $\diff\ge\ldiff\ge 0$ almost everywhere in $\Omega$. The case of locally heterogeneous anisotropic diffusion can be considered as well using the ideas in~\cite{Di-Pietro.Ern:14b}.
For the advective velocity $\vel:\Omega\to\Real^d$, we assume the regularity $\vel\in{\rm Lip}(\Omega)^d$, and for the reaction coefficient $\mu:\Omega\to\Real$, we assume $\mu\in L^\infty(\Omega)$ and that $\mu$ is bounded from below by a real number $\mu_0>0$. For simplicity, we work under the assumption $\DIV\vel\equiv0$; the case $\DIV\vel\not\equiv 0$ can be treated similarly, provided $\reac+\frac12\DIV\vel\ge\reac_0>0$.
At the continuous level with non-degenerate diffusion, this assumption can be relaxed
to $\reac\ge 0$ (no condition on $\DIV\vel$), see \cite{Droniou:02}; the locally degenerate case and the analysis of discretization schemes is, however, more delicate. Some numerical tests (not reported here) with $\reac_0=0$ indicate that the present scheme remains well-behaved.
We introduce the following sets (cf. Figure~\ref{fig:dpeg} below for an illustration):
\begin{subequations}
\begin{align}
  \label{eq:Gamma}
  \Gam &\eqbydef\left\{
  \vec{x}\in\partial\Omega\st\text{$\diff(\vec{x})>0$ or $(\vel\SCAL\normal)(\vec{x})<0$}
  \right\}, \\
  \Int[\pm] &\eqbydef \{ \vec{x} \in \Intnu \st \pm (\vel\SCAL\normal_I)(\vec{x}) > 0 \},
\end{align}
\end{subequations}
where $\Intnu$ is the diffusive/nondiffusive interface and $\normal_I$ is the unit normal to $\Intnu$ pointing out of the diffusive region. More precisely, $\Intnu$ is the set of points in $\Omega$ located at an interface between two distinct subdomains $\Omega_i$ and $\Omega_j$ of $P_\Omega$ such that $\restrto{\diff}{\Omega_i}>\restrto{\diff}{\Omega_j}=0$. We assume that 
\begin{equation*} \label{eq:hyp.beta.cross}
\text{$(\vel\SCAL\normal_I)(\vec{x})\ne0$ for a.e. $\vec{x}\in\Intnu$}.
\end{equation*}
For given source term $f\in L^2(\Omega)$ and boundary datum $g\in L^2(\Gam)$, the continuous problem reads
\begin{subequations}
  \label{eq:strong}
  \begin{alignat}{2}
    \label{eq:strong:PDE}
    \DIV(-\diff\GRAD u + \vel u) + \reac u &= f &\qquad& \text{in $\Omega\setminus\Intnu$,} 
    \\
    \label{eq:match.flux}
    \jump{-\diff\GRAD u + \vel u}\SCAL\normal_I&=0 &\quad&\text{on $\Intnu$},
    \\
    \label{eq:match.u}
    \jump{u}&=0 &\quad&\text{on $\Int[+]$},
    \\
    \label{eq:strong:BC}
    u &= g &\qquad& \text{on $\Gam$,}
  \end{alignat}
\end{subequations}
where $\jump{\cdot}$ denotes the jump across $\Intnu$ (the sign is irrelevant). Notice that the boundary condition is enforced at portions of the boundary touching a diffusive region or a nondiffusive region provided the advective field flows into the domain. 
A weak formulation for~\eqref{eq:strong} has been analyzed in~\cite{Di-Pietro.Ern.ea:08}. In the non-degenerate case $\ldiff>0$, $\Gam=\partial\Omega$ and the usual weak formulation in the space $H_0^1(\Omega)$ holds, cf., e.g.,~\cite[Section~4.6.1]{Di-Pietro.Ern:12}.%

\section{Discrete setting}
\label{sec:setting}
This section presents the discrete setting: admissible mesh sequences, analysis tools on such meshes, DOFs, reduction maps, and reconstruction operators.

\subsection{Assumptions on the mesh}

Denote by ${\cal H}\subset \Real_*^+ $ a countable set of meshsizes having $0$ as its unique accumulation point.
Following~\cite[Chapter~4]{Di-Pietro.Ern:12}, we consider $h$-refined mesh sequences $(\Th)_{h \in {\cal H}}$ where, for all $ h \in {\cal H} $, $\Th$ is a finite collection of nonempty disjoint open polyhedral elements $T$
such that $\closure{\Omega}=\bigcup_{T\in\Th}\closure{T}$ and $h=\max_{T\in\Th} h_T$
with $h_T$ standing for the diameter of the element $T$.
A face $F$ is defined as a hyperplanar closed connected subset of $\closure{\Omega}$ with positive $ (d{-}1) $-dimensional Hausdorff measure and such that
\begin{inparaenum}[(i)]
  \item either there exist $T_1(F),T_2(F)\in\Th $ such that $F\subset\partial T_1(F)\cap\partial T_2(F)$ and $F$ is called an interface or 
  \item there exists $T(F)\in\Th$ such that $F\subset\partial T(F)\cap\partial\Omega$ and $F$ is called a boundary face.
\end{inparaenum}
In what follows, the dependence on $F$ of $T_1(F)$ and $T_2(F)$ (when $F$ is an interface) and of $T(F)$ (when $F$ is a boundary face) is omitted when no ambiguity can arise.
Interfaces are collected in the set $\Fhi$, boundary faces are collected in $\Fhb$, and we let $\Fh\eqbydef\Fhi\cup\Fhb$.
The diameter of a face $F\in\Fh$ is denoted by $h_F$.
For all $T\in\Th$, $\Fh[T]\eqbydef\{F\in\Fh\st F\subset\partial T\}$ denotes the set of faces contained in $\partial T$ (with $\partial T$ denoting the boundary of $T$) and, for all $F\in\Fh[T]$, $\normal_{TF}$ is the unit normal to $F$ pointing out of $T$.
Symmetrically, for all $F\in\Fh$, we let $\Th[F]\eqbydef\{T\in\Th\st F\subset \partial T\}$ be
the set of elements having $F$ as a face.
For each interface $F\in\Fhi$, we fix an orientation as follows: we select a fixed ordering for the elements $T_1,T_2\in\Th$ such that $F\subset\partial T_1\cap\partial T_2$ and we let $\normal_F\eqbydef\normal_{T_1,F}$. 
For a boundary face, we simply take $\normal_F=\normal$, the outward unit normal to $\Omega$.

Our analysis hinges on the following two assumptions on the mesh sequence. 

\begin{assumption}[Admissible mesh sequence]
  \label{def:adm.Th}
   For all $h\in{\cal H}$, $\Th$ admits a matching simplicial submesh $\fTh$ such that any cell and any face in $\fTh$ belongs to only one cell and face of $\Th$, respectively, and there exists a real number $\varrho>0$ independent of $h$ such that, for all $h\in{\cal H}$,
  \begin{inparaenum}[(i)]
  \item for all simplex $S\in\fTh$ of diameter $h_S$ and inradius $r_S$, $\varrho h_S\le r_S$ and
  \item for all $T\in\Th$, and all $S\in\fTh$ such that $S\subset T$, $\varrho h_T \le h_S$.
  \end{inparaenum}
\end{assumption}

\begin{assumption}[Compatible mesh sequence] \label{def:comp.Th}
  \begin{inparaenum}[(i)]
  \item Any mesh cell belongs to one and only one subdomain $\Omega_i$ of the partition $P_\Omega$;
  \item Any mesh face having an intersection with the interface $\Int$ (of positive $(d-1)$-dimensional Hausdorff measure) is included in one of the two sets $\Int[\pm]$;
  \item In any mesh face such that the diffusion coefficient vanishes on both of its sides, the normal component of $\vel$ is nonzero in a subset of positive measure. 
  \end{inparaenum}
\end{assumption}

The simplicial submesh in Assumption~\ref{def:adm.Th} is just a theoretical tool used to prove the results in Section~\ref{sec:setting:tools}, and it is not used in the actual construction of the discretization method. Furthermore, 
a straightforward consequence of Assumption~\ref{def:comp.Th}(i) is that $\diff$ is piecewise constant on $\Th$. Assumption~\ref{def:comp.Th}(ii) is important in the error analysis so that the face unknowns on $\Int$ capture the exact solution from the diffusive side. In practice, this assumption is not restrictive since the faces of the original mesh can be split to satisfy Assumption~\ref{def:comp.Th}(ii).
Assumption~\ref{def:comp.Th}(iii) can be avoided by adding some crosswind diffusion to the stabilization of the advective-reactive bilinear form in the spirit of a Lax--Friedrichs flux, so that the difference between cell- and face-based DOFs is always penalized on faces included in the nondiffusive region.

\subsection{Analysis tools}
\label{sec:setting:tools}

We recall some results that hold uniformly in $h$ on admissible mesh sequences.
In what follows, for $X\subset\closure{\Omega}$, we denote by ${(\cdot,\cdot)}_X$ and $\norm[X]{{\cdot}}$ the standard inner product and norm in $L^2(X)$, respectively, with the convention that the subscript is omitted whenever $X=\Omega$. The same notation is used in the vector-valued case $L^2(X)^d$.
According to~\cite[Lemma 1.42]{Di-Pietro.Ern:12}, for all $h\in{\cal H}$, all $T\in\Th$, and all $F\in\Fh[T]$, $h_F$ is comparable to $h_T$ in the sense that
\begin{equation}
  \label{eq:hF}
  \varrho^2 h_T\leq h_F\leq h_T.
\end{equation}
Moreover, \cite[Lemma 1.41]{Di-Pietro.Ern:12} shows that there exists an integer $\Np$ depending on $\varrho$ such that
\begin{equation}
  \label{eq:Np}
  \forall h\in{\cal H},\qquad
  \max_{T\in\Th}\card{\Fh[T]}\le\Np.
\end{equation}
Let $l\ge0$ be a nonnegative integer. 
For an $n$-dimensional subset $X$ of $\closure{\Omega}$ ($n\leq d$), $\Poly[n]{l}(X)$ is the space spanned by the restrictions to $X$ of $n$-variate polynomials of total degree $\le l$. 
Then, there exists a real number $C_{\rm tr}$ depending on $\varrho$ and $l$, but independent of $h$, such that the following discrete trace inequality holds for all $T\in\Th$ and $F\in\Fh[T]$, cf.~\cite[Lemma~1.46]{Di-Pietro.Ern:12}:
\begin{equation}
  \label{eq:trace.disc}
  \norm[F]{v} \le 
  C_{\rm tr} h_F^{- \nicefrac12} \norm[T]{v}
  \qquad\forall v\in\Poly{l}(T).
\end{equation}
Furthermore, the following inverse inequality holds for all $T\in\Th$ with $C_{\rm inv}$ again depending
on $\varrho$ and $l$, but independent of $h$,
cf.~\cite[Lemma~1.44]{Di-Pietro.Ern:12}, 
\begin{equation}
  \label{eq:inv}
  \norm[T]{\GRAD v}\le 
  C_{\rm inv} h_T^{-1}\norm[T]{v}
  \qquad\forall v\in\Poly{l}(T).
\end{equation}
Moreover, using~\cite[Lemma~1.40]{Di-Pietro.Ern:12} together with the results of~\cite{Dupont.Scott:80}, one can prove that there exists a real number $C_{\rm app}$ depending on $\varrho$ and $l$, but independent of $h$, such that, for all $T\in\Th$, denoting by $\lproj[T]{l}$ the $L^2$-orthogonal projector on $\Poly{l}(T)$, the following holds:
For all $s\in\{1,\ldots,l+1\}$ and all $v\in H^s(T)$, 
\begin{equation}
  \label{eq:approx.lproj.T}
  \seminorm[H^m(T)]{v - \lproj[T]{l} v }
  + h_T^{\nicefrac12}\seminorm[H^m(\partial T)]{v-\lproj[T]{l} v}
  \le 
  C_{\rm app} h_T^{s-m} 
  \seminorm[H^s(T)]{v}
  \qquad \forall m \in \{0,\ldots,s-1\}.
\end{equation}

\subsection{Degrees of freedom, interpolation, and reconstruction}
\label{sec:DOFs}

Let a polynomial degree $k\ge 0$ be fixed.
For all $T\in\Th$, the local space of DOFs is
\begin{equation*}
  \label{eq:UT}
  \UT\eqbydef\Poly{k}(T)\times\left\{
  \bigtimes_{F\in\Fh[T]}\Poly[d-1]{k}(F)
  \right\},
\end{equation*}
and we use the notation $\sv=(\unv[T],(\unv[F])_{F\in\Fh[T]})$ for a generic element $\sv\in\UT$.
We define the local interpolation operator $\IT: H^1(T)\to\UT$ such that, for all $v\in H^1(T)$,
\begin{equation*}
  \label{eq:IT}
  \IT v \eqbydef \left(\lproj[T]{k}v, (\lproj[F]{k}v)_{F\in\Fh[T]}\right),
\end{equation*}
where $\lproj[F]{k}$ denotes the $L^2$-orthogonal projector on $\Poly[d-1]{k}(F)$.
Following~\cite{Di-Pietro.Ern.ea:14}, for all $T\in\Th$, we define the local potential reconstruction operator $\pT:\UT\to\Poly{k+1}(T)$ such that, for all $\sv\eqbydef(\unv[T],(\unv[F])_{F\in\Fh[T]})\in\UT$,
\begin{equation}
  \label{eq:pT}
  \begin{alignedat}{2}
    (\GRAD\pT\sv,\GRAD w)_T
    &= (\GRAD\unv[T],\GRAD w)_T
    + \sum_{F\in\Fh[T]}(\unv[F]-\unv[T],\GRAD w\SCAL\normal_{TF})_F
    &\qquad&\forall w\in\Poly{k+1}(T),
    \\
    \int_T\pT\sv &= \int_T\unv[T].
  \end{alignedat}
\end{equation}
The discrete Neumann problem~\eqref{eq:pT} is well-posed.
The following result has been proved in~\cite[Lemma~3]{Di-Pietro.Ern.ea:14}.
\begin{lemma}[Approximation properties for $\pT\IT$] \label{lem:approx.rT}
  There exists a real number $C>0$, depending on $\varrho$ and $k$, but independent of $h_T$, such that, for all $v\in H^{k+2}(T)$,
  \begin{multline}
    \label{eq:rT.approx}
    \norm[T]{v-\pT\IT v}
    + h_T^{\nicefrac12}\norm[\partial T]{v-\pT\IT v} \\
    + h_T\norm[T]{\GRAD (v-\pT\IT v)}
    + h_T^{\nicefrac32}\norm[\partial T]{\GRAD (v-\pT\IT v)}
    \le C h_T^{k+2}\norm[H^{k+2}(T)]{v}.
  \end{multline}
\end{lemma}

\section{Local bilinear forms}
\label{sec:disc.op}

In this section we define the local bilinear forms. These forms are expressed in terms of local DOFs and are instrumental in deriving the discrete problem in Section~\ref{sec:method}. 

\subsection{Diffusion}\label{sec:diff}

To discretize the diffusion term in~\eqref{eq:strong}, we introduce, for all $T\in\Th$, the bilinear form $a_{\diff,T}$ on $\UT\times\UT$ such that
\begin{equation}
    \label{eq:aT.nu}
    a_{\diff,T}(\sw,\sv)
    \eqbydef (\diff[T]\GRAD\pT\sw,\GRAD\pT\sv)_T + s_{\diff,T}(\sw,\sv),
\end{equation}
with stabilization bilinear form $s_{\diff,T}$ on $\UT\times\UT$ such that
\begin{equation}
  \label{eq:sT.nu}
  s_{\diff,T}(\sw,\sv)\eqbydef
  \sum_{F\in\Fh[T]}
  \frac{\diff[T]}{h_F}(\lproj[F]{k}(\unw[F]-\PT\sw),\lproj[F]{k}(\unv[F]-\PT\sv))_F.
\end{equation}
In~\eqref{eq:sT.nu}, the potential reconstruction $\PT:\UT\to\Poly{k+1}(T)$ is such that, for all $\sv\in\UT$, 
\begin{equation*}
  \label{eq:RT}
  \PT\sv \eqbydef  \unv[T] + (\pT\sv - \lproj[T]{k}\pT\sv),
\end{equation*}
where the second term can be interpreted as a high-order correction of $\unv[T]$.

\subsection{Advection-reaction}
\label{sec:disc.prob:adv.rea}

For all $T\in\Th$, we introduce the discrete advective derivative $\GBT:\UT\to\Poly{k}(T)$ such that, for all $\sv\in\UT$ and all $w\in\Poly{k}(T)$,
\begin{subequations}
  \begin{align}
    \label{eq:GBT}
    (\GBT\sv,w)_T 
    &= -(\unv[T],\vel\SCAL\GRAD w)_T + \sum_{F\in\Fh[T]}((\vel\SCAL\normal_{TF})\unv[F],w)_F
    \\
    \label{eq:GBT.bis}
    &= (\vel\SCAL\GRAD\unv[T], w)_T + \sum_{F\in\Fh[T]}((\vel\SCAL\normal_{TF})(\unv[F]-\unv[T]),w)_F,
  \end{align}
\end{subequations}
where we have integrated by parts the first term in the right-hand side and used $\DIV\vel\equiv 0$ to pass from~\eqref{eq:GBT} to~\eqref{eq:GBT.bis}.
We introduce, for all $T\in\Th$, the local bilinear form $a_{\vel,\reac,T}$ on $\UT\times\UT$ such that
\begin{equation*}
  \label{eq:aT.vel.mu}
  a_{\vel,\reac,T}(\sw,\sv)\eqbydef
  -(\unw[T],\GBT\sv)_T + (\reac\unw[T],\unv[T])_T + s_{\vel,T}\pneg(\sw,\sv).
\end{equation*}
The local stabilization bilinear forms $s_{\vel,T}^\pm$ on $\UT\times\UT$ are such that
\begin{equation*}
  \label{eq:sT.vel}
  s_{\vel,T}^\pm(\sw,\sv)\eqbydef
  \sum_{F\in\Fh[T]}(\tfrac{\diff[F]}{h_F}\Apm(\Pe_{TF})(\unw[F]-\unw[T]),\unv[F]-\unv[T])_F.
\end{equation*}
Here, $\diff[F]\eqbydef\min_{T\in\Th[F]}\diff[T]$ is the lowest diffusion coefficient from the (one or) two cells sharing $F$, and
the local (oriented) P\'eclet number $\Pe_{TF}$ is defined if $\diff[F]>0$ by
\begin{equation}
  \label{def:Peclet}
  \Pe_{TF}=h_F \frac{\vel\SCAL\normal_{TF}}{\diff[F]},
\end{equation}
while we use~\eqref{eq:def:Apmdeg} below if $\diff[F]=0$. 
Since, for all $F\in\Fhb$, there is a unique $T\in\Th$ such that $F\subset\partial T$, we simply write $\Pe_F$ instead of $\Pe_{TF}$ in this case. Notice that the local P\'eclet number $\Pe_{TF}$ is a function $F\to\mathbb{R}$.

The functions $\Apm:\Real\rightarrow\Real$ are such that $\Apm(s)=\frac12(\Aabs(s)\pm s)$ for all $s\in\Real$, and the function $\Aabs:\Real\rightarrow\Real$ is assumed to satisfy the following design conditions: 
\begin{description}
\item \ASSUM{A}{1} $\Aabs$ is a Lipschitz-continuous function such that $\Aabs(0)=0$ and,  for all $s\in\Real$, $\Aabs(s)\ge0$ and $\Aabs(-s)=\Aabs(s)$;
\item \ASSUM{A}{2} there exists $\underline{a}\ge 0$ such that $\Aabs(s)
\ge \underline{a}|s|$ for any $|s|\ge 1$;
\item \ASSUM{A}{3} If $\ldiff=0$, $\lim_{s\to+\infty}\frac{\Aabs(s)}{s}= 1$ and,
consistently with~\ASSUM{A}{1}, $\lim_{s\to -\infty}\frac{\Aabs(s)}{s}= -1$. Coherently, for all $T\in\Th$ and all $F\in\Fh[T]$ such that $\diff[F]=0$, we set 
  \begin{equation}\label{eq:def:Apmdeg}
  \tfrac{\diff[F]}{h_F}\Apm(\Pe_{TF})\eqbydef \lim_{\diff\to 0^+}\left(\tfrac{\diff}{h_F}
\Apm\left(\tfrac{h_F}{\diff}\vel\SCAL\normal_{TF}\right)\right)=(\vel\SCAL\normal_{TF})^{\pm},
  \end{equation}
where, for a real number $s$, we have denoted $s^\pm\eqbydef\tfrac12(|s|\pm s)$.
\end{description}

As already pointed out in \cite{Beirao-da-Veiga.Droniou.ea:10,Chainais.Droniou:11,Droniou:10},
using the generic functions $\Apm$ in the definition of the advective
terms allows for a unified treatment of several classical discretizations. 
The centered scheme corresponds to $\Aabs(s)=0$, which fails to satisfy \ASSUM{A}{2}-\ASSUM{A}{3}. Instead,
Properties \ASSUM{A}{1}--\ASSUM{A}{3} are fulfilled by the following methods:
\begin{itemize}
\item \emph{Upwind scheme}: $\Aabs(s)=|s|$ (so that $\Apm(s)=s^\pm$
and $\frac{\diff[F]}{h_F}\Apm(\Pe_{TF})=(\vel\SCAL\normal_{TF})^{\pm}$). 
\item \emph{Locally upwinded $\theta$-scheme}: $\Aabs(s)=(1-\theta(s))|s|$, where $\theta\in C^1_c(-1,1)$, $0\le \theta\le 1$ and
$\theta\equiv 1$ on $[-\nicefrac12,\nicefrac12]$, corresponding to the centered scheme if $s\in[-\nicefrac12,\nicefrac12]$ (dominating diffusion)
and the upwind scheme if $s\ge 1$ (dominating advection). 
\item \emph{Scharfetter--Gummel scheme}: $\Aabs(s)=2\left(\frac{s}2\coth(\frac{s}2)-1\right)$.
\end{itemize}
The advantage of the locally upwinded $\theta$-scheme and the Scharfetter--Gummel scheme over the upwind scheme is that they behave as the centered scheme, and thus introduce less artificial diffusion, when $s$ is close to zero (dominating diffusion).

\begin{remark}[Assumption \ASSUM{A}{2}] \label{rem:A2} This assumption implies that $|\Apm(s)|\le \overline{a}|A(s)|$ for all $|s|\ge1$ with $\overline{a}=\frac12(1+\underline{a}^{-1})$. Furthermore, the threshold $|s|\ge 1$ is arbitrary. If it is
changed into $|s|\ge \underline{b}$ for some fixed $\underline{b}\ge 1$, then the only
modification in the error estimate~\eqref{eq:conv.rate} below is to change the
term $\min(1,\Pe_T)$ into $\min(\underline{b},\Pe_T)$.
\end{remark}

\begin{remark}[Assumption \ASSUM{A}{3}]
  \label{rem:A5}
  Assumption \ASSUM{A}{3} is required only in the locally degenerate case where the diffusion coefficient vanishes in one part of the domain.
\end{remark}


\section{Discrete problem and main results}
\label{sec:method}

In this section we build the discretization of~\eqref{eq:strong} using the local bilinear forms of Section~\ref{sec:disc.op}.
A key point is the weak enforcement of boundary conditions to achieve robustness with respect to the P\'eclet number.

\subsection{Discrete bilinear forms}

Local DOFs are collected in the following global space obtained by patching interface values:
\begin{equation*}
  \label{eq:tUh}
  \Uh\eqbydef\left\{ \bigtimes_{T\in\Th}\Poly{k}(T) \right\}\times
  \left\{ \bigtimes_{F\in\Fh}\Poly[d-1]{k}(F) \right\}.
\end{equation*}
We use the notation $\sv[h]=((\unv[T])_{T\in\Th},(\unv[F])_{F\in\Fh})$ for a generic element $\sv[h]\in\Uh$ and, for all $T\in\Th$, it is understood that $\sv[T]$ denotes the restriction of $\sv[h]$ to $\UT$.

Denoting by $\varsigma>0$ a user-dependent boundary penalty parameter, we define the global diffusion bilinear form $a_{\diff,h}$ on $\Uh\times\Uh$ such that
\begin{equation}
  \label{eq:ah.nu}
  a_{\diff,h}(\sw[h],\sv[h])
  \eqbydef
  \sum_{T\in\Th}a_{\diff,T}(\sw,\sv)
  +\sum_{F\in\Fhb} \left\{
  -(\diff[F]\GRAD\pTF\sw\SCAL\normal_{TF},\unv[F])_F
  + \frac{\varsigma\diff[F]}{h_F}(\unw[F],\unv[F])_F
  \right\},
\end{equation}
and the global advection-reaction bilinear form $a_{\vel,\reac,h}$ such that
\begin{equation}
  \label{eq:ah.vel.mu}
  a_{\vel,\reac,h}(\sw[h],\sv[h])\eqbydef
  \sum_{T\in\Th}a_{\vel,\reac,T}(\sw,\sv)
  +\sum_{F\in\Fhb}(\tfrac{\diff[F]}{h_F}\Ap(\Pe_F)\unw[F],\unv[F])_F.
\end{equation}
The rightmost terms in~\eqref{eq:ah.nu} and~\eqref{eq:ah.vel.mu} are responsible for the weak enforcement of the boundary condition on $\Gam$.
Finally, we set
\begin{equation}
  \label{eq:ah}
  a_h(\sw[h],\sv[h])
  \eqbydef a_{\diff,h}(\sw[h],\sv[h]) + a_{\vel,\reac,h}(\sw[h],\sv[h]),
\end{equation}
and we define the linear form $l_h$ on $\Uh$ such that
\begin{equation}
  \label{eq:lh}
  l_h(\sv[h])
  \eqbydef
  \sum_{T\in\Th}(f,\unv[T])_T
  +\sum_{F\in\Fhb}\left\{
  (\tfrac{\diff[F]}{h_F}\Am(\Pe_F)g,\unv[F])_F
  +\frac{\varsigma\diff[F]}{h_F}(g,\unv[F])_F
  \right\}.
\end{equation}
The discrete problem reads:
Find $\su[h]\in\Uh$ such that, for all $\sv[h]\in\Uh$, 
\begin{equation}
  \label{eq:discrete}
  a_h(\su[h],\sv[h]) = l_h(\sv[h]).
\end{equation}

\begin{remark}[Symmetric variation for $a_{\diff,h}$]
  \label{rem:ah.nu.symm}
  A symmetric expression of $a_{\diff,h}$ is obtained by adding the term
  $
  -\sum_{F\in\Fhb}(\unw[F],\diff[F]\GRAD\pTF\sv\SCAL\normal_{TF})_F,
  $
  to the right-hand side of~\eqref{eq:ah.nu} and, correspondingly, the term
  $
  -\sum_{F\in\Fhb}(g,\diff[F]\GRAD\pTF\sv\SCAL\normal_{TF})_F
  $
  to the right-hand side of~\eqref{eq:lh}.
  This variation is not further pursued here since the problem~\eqref{eq:strong} is itself nonsymmetric.
\end{remark}%

\begin{remark}[Static condensation]
It is possible to locally eliminate the degrees of freedom inside each cell
$T\in\Th$ by selecting in \eqref{eq:discrete} a test function $\sv[h]$ such that $\unv[T']=0$
for all $T'\neq T$ and $\unv[F]=0$ for all $F\in \Fh$. This yields $a_{\diff,T}(\su,\sv) + a_{\vel,\reac,T}(\su,\sv) = (f,\unv[T])$ for all 
$\unv[T]\in \Poly{k}(T)$.
This relation involves only $\unu[T]$ and $(\unu[F])_{F\in\Fh[T]}$ and,
for fixed face values, it is a square system in $\unu[T]$ with a right-hand side
defined through $f$ and $(\unu[F])_{F\in\Fh[T]}$. 
The invertibility of this system follows from the fact that
$a_{\diff,T}(\su,\su)\ge 0$ and that
\[
	a_{\vel,\reac,T}(\su,\su)=\sum_{F\in\Fh[T]}([\tfrac{1}{2}\vel\SCAL\normal_{TF}+\tfrac{\diff[F]}{h_F}\Am(\Pe_{TF})]\unu[T],\unu[T])_F +(\reac\unu[T],\unu[T])_T,
\]
where we used Stokes formula for the advective derivative.
If $\diff[F]>0$ then, recalling the definition $\Am(s)=\tfrac{1}{2}\Aabs(s)-\tfrac{1}{2}s$, we infer that
\[
\tfrac{1}{2}\vel\SCAL\normal_{TF}+\tfrac{\diff[F]}{h_F}\Am(\Pe_{TF})
=\tfrac{\diff[F]}{h_F}[\tfrac{1}{2}\Pe_{TF}+\Am(\Pe_{TF})]
=\tfrac{1}{2}\tfrac{\diff[F]}{h_F}\Aabs(\Pe_{TF})\ge 0.
\]
This relation still holds if $\diff[F]=0$ provided that we use the definition
\eqref{eq:def:Apmdeg} for $\Am$ and $\Aabs$. Hence, 
$a_{\vel,\reac,T}(\su,\su)\ge (\reac\unu[T],\unu[T])_T$, and we conclude using 
$\reac\ge \reac_0>0$.
\end{remark}

\subsection{Discrete norms and stability}

The analysis of the discrete problem~\eqref{eq:discrete} involves several norms.
For the sake of easy reference, their definitions are gathered here, as well
as some related stability properties.
The energy-like diffusion (semi)norm is defined on $\Uh$ by
\begin{equation}
  \label{eq:norm.nu}
  \begin{array}{l}
	\displaystyle\norm[\diff,h]{\sv[h]}^2 \eqbydef 
  \sum_{T\in\Th} \norm[\diff,T]{\sv}^2
+ \seminorm[\diff,\partial\Omega]{\sv[h]}^2
\quad\mbox{ with:} \\
	\displaystyle \norm[\diff,T]{\sv}^2\eqbydef a_{\diff,T}(\sv,\sv)
\quad\mbox{ and }\quad\seminorm[\diff,\partial\Omega]{\sv[h]}^2\eqbydef
  \sum_{F\in\Fhb} \frac{\diff[F]}{h_F}\norm[F]{\unv[F]}^2,
	\end{array}
\end{equation}
and owing to~\cite[Lemma~3.1]{Di-Pietro.Ern.ea:14}, we observe that there is $\eta>0$, depending only on $\varrho$, $d$, and $k$, such that, for all $\sv\in\UT$,
\begin{equation}\label{eq:equiv:norms}
	\eta \norm[\diff,T]{\sv}^2\le  \diff[T]\norm[T]{\GRAD\unv[T]}^2
+ \sum_{F\in\Fh[T]}\frac{\diff[T]}{h_T}\norm[F]{\unv[F]-\unv[T]}^2
\le \eta^{-1} \norm[\diff,T]{\sv}^2.
\end{equation}
The advection-reaction (semi)norm is defined on $\Uh$ by
\begin{equation}
  \label{eq:norm.vel.mu}
	\begin{array}{l}
  \displaystyle\norm[\vel,\reac,h]{\sv[h]}^2  \eqbydef\sum_{T\in\Th}\norm[\vel,\reac,T]{\sv}^2
  +\seminorm[\vel,\partial\Omega]{\sv[h]}^2\quad\mbox{ with:}\\
  \displaystyle	\norm[\vel,\reac,T]{\sv}^2\eqbydef\frac12\sum_{F\in\Fh[T]}\norm[F]{\big[\tfrac{\diff[F]}{h_F}\Aabs(\Pe_{TF})\big]^{\nicefrac12}(\unv[F]-\unv[T])}^2+\tc^{-1}\norm[T]{\unv[T]}^2
	\quad\mbox{ and }\\
	\displaystyle
\seminorm[\vel,\partial\Omega]{\sv[h]}^2\eqbydef
  \frac12\sum_{F\in\Fhb}\norm[F]{\big[\tfrac{\diff[F]}{h_F}\Aabs(\Pe_{F})\big]^{\nicefrac12}\unv[F]}^2.
	\end{array}
\end{equation}
Following~\cite[Chapter~2]{Di-Pietro.Ern:12}, the reference time $\tc$ and velocity $\velc$ are defined by
\begin{equation}
  \label{eq:Lvel.tc.velc}
  \Lvel\eqbydef\max_{1\le i\le d} \norm[L^\infty(T)^d]{\GRAD\beta_i}\,,\quad
  \tc\eqbydef\{\max(\norm[L^\infty(T)]{\mu},\Lvel)\}^{-1}\,,
  \quad
  \velc\eqbydef\norm[L^\infty(T)^d]{\vel}\,,
\end{equation}
(recall that $\vel\in{\rm Lip}(\Omega)^d$ implies
$\vel\in W^{1,\infty}(\Omega)^d$).
Finally, we define two advection-diffusion-reaction norms on $\Uh$ as follows:
\begin{equation}
  \label{eq:norm.h}
  \norm[\flat,h]{\sv[h]}^2
  \eqbydef\norm[\diff,h]{\sv[h]}^2
  +\norm[\vel,\reac,h]{\sv[h]}^2\quad\mbox{ and } \quad
  \norm[\sharp,h]{\sv[h]}^2
  \eqbydef\norm[\flat,h]{\sv[h]}^2
  +\sum_{T\in\Th} h_T\velc^{-1}\norm[T]{\GBT\sv}^2,
\end{equation}
where the summand is taken only if $\velc\ne0$.
The error estimate stated in Theorem~\ref{thm:err.est} below uses the $\norm[\sharp,h]{{\cdot}}$-norm, and therefore delivers information on the advective derivative of the error, which is important in the advection-dominated regime.
The $\norm[\flat,h]{{\cdot}}$-norm is, on the other hand, the natural coercivity norm for the bilinear form $a_h$, and is used as an intermediary step in the error analysis. The coercivity norm is sufficient for the error analysis in the diffusion-dominated regime.

\begin{remark}[Norms]
  \label{rem:norms}
  Owing to Assumption~\ref{def:comp.Th}(iii), we infer that $\diff[F]\neq 0$ or $\vel\SCAL\normal_F\not\equiv 0$ (on a subset with positive measure). Hence, $\norm[\flat,h]{{\cdot}}$ and $\norm[\sharp,h]{{\cdot}}$ are norms on $\Uh$. Indeed, if $\diff[F]\neq 0$, then by \eqref{eq:equiv:norms} the diffusive norm
controls the term $\unv[F]-\unv[T]$ and, if $\diff[F]=0$, owing to \eqref{eq:def:Apmdeg}, we obtain $\frac{\diff[F]}{h_F}\Aabs(\Pe_{TF})=|\vel\SCAL\normal_{TF}|\not\equiv 0$ and the advective norm controls $\unv[F]-\unv[T]$.
\end{remark}

Our first important result concerns stability. The proof is postponed to Section~\ref{sec:proof_stab}.
\begin{lemma}[Stability of $a_h$]
  \label{lem:stab.ah}
  Assume $\varsigma\ge1+\tfrac{C_{\rm tr}^2\Np}{2}$ and \ASSUM{A}{1}--\ASSUM{A}{3}. Then, for all $\sv[h]\in\Uh$, the following holds:
\begin{equation}
  \label{eq:coer.ah}
  \xi\norm[\flat,h]{\sv[h]}^2\le a_h(\sv[h],\sv[h]),
\end{equation}
with $\xi\eqbydef\min_{T\in\Th}(\frac12,\tc\reac_0)>0$. 
Assume additionally that, for all $T\in\Th$,
  \begin{equation}
    \label{eq:err.est.ass}
    h_T\Lvel\le\velc\quad\mbox{ and }\quad
    h_T\norm[L^\infty(T)]{\reac} \le\velc,
  \end{equation}
  where $\Lvel$, $\velc$, and $\tc$ are defined by \eqref{eq:Lvel.tc.velc}.
  Then, there exists a real number $\gamma>0$, independent of $h$, $\diff$, $\vel$, and $\reac$, such that, for all $\sw[h]\in\Uh$,
  \begin{equation}
    \label{eq:inf-sup.ah}
    \gamma\xi\norm[\sharp,h]{\sw[h]}\le 
    \sup_{\sv[h]\in\Uh\setminus\{\mathsf{0}\}}\frac{a_h(\sw[h],\sv[h])}{\norm[\sharp,h]{\sv[h]}}.
  \end{equation}
\end{lemma}

\begin{remark}[Threshold for $\varsigma$]
  The dependency on $C_{\rm tr}$ of the threshold on $\varsigma$ introduced in Lemma~\ref{lem:stab.ah} can be removed by considering a lifting-based penalty term such as the one discussed in~\cite[Section~5.3.2]{Di-Pietro.Ern:12} and originally introduced by Bassi et al.~\cite{Bassi.Rebay.ea:97} in the context of discontinuous Galerkin methods. Furthermore, the strict minimal threshold in Lemma~\ref{lem:stab.ah} is $\varsigma>\tfrac{C_{\rm tr}^2\Np}{4}$. It is also possible to replace $\Np$ by the maximum number of faces of cells having a boundary face.
\end{remark}

\begin{remark}[Assumption~\eqref{eq:err.est.ass}]
  The first inequality in~\eqref{eq:err.est.ass} stipulates that the meshsize resolves the spatial variations of the advective velocity $\vel$.
The quantity $\Da_T\eqbydef h_T\norm[L^\infty(T)]{\mu}\velc^{-1}$ is a local Damk\"{o}hler number relating the reactive and advective time scales. The second inequality in \eqref{eq:err.est.ass}
assumes that $\Da_T\le 1$ for all $T\in\Th$, meaning that we are not concerned with the reaction-dominated regime. We could also state a
stability result without \eqref{eq:err.est.ass}, but the dependency on the various constants would be somewhat more intricate.
\end{remark}

\subsection{Error estimate}
\label{sec:res.err}

For all $F\in\Fh$, we denote by $\Tdiff$ one element of $\Th[F]$ such that $\Tdiff\in\arg\max_{T\in\Th[F]}\diff[T]$ (such an element may not be unique when $F$ is an interface).
Consider now an interface $F\in\Fhi$ such that $F\subset\Int[-]$.
Since the exact solution can jump on $F$, we have to deal with a possibly two-valued trace for the exact solution.
It turns out that, in this case, the face unknown captures the trace from the diffusive side, i.e., from the unique element $\Tdiff\in\Th[F]$ such that $\restrto{\diff}{\Tdiff}>0$.
We therefore define the global interpolation operator $\Ih: H^1(\Omega\setminus\Int)\to\Uh$ such that, for all $v\in H^1(\Omega\setminus\Int)$,
\begin{equation}
  \label{eq:Ih}
  \Ih v \eqbydef \left( (\lproj[T]{k} v)_{T\in\Th}, (\lproj[F]{k}  [\restrto{v}{T_{\diff}(F)}])_{F\in\Fh}\right).
\end{equation}%
Our main result is the following estimate on the discrete approximation error $(\Ih u-\su[h])$ measured in the $\norm[\sharp,h]{{\cdot}}$-norm. The proof is postponed to Section~\ref{sec:proof_err}.

\begin{theorem}[Error estimate]
  \label{thm:err.est}
Assume $\varsigma\ge1+\tfrac{C_{\rm tr}^2\Np}{2}$, \ASSUM{A}{1}--\ASSUM{A}{3}, and~\eqref{eq:err.est.ass}. Denote by $u$ and $\su[h]$ the unique solutions to~\eqref{eq:strong} and~\eqref{eq:discrete}, respectively, and assume that $\restrto{u}{T}\in H^{k+2}(T)$ for all $T\in\Th$. Then, there exists a real number $\gamma'>0$ 
depending on $\varrho$, $d$, and $k$, but independent of
$h$, $\diff$, $\vel$, and $\reac$, such that, letting $\shu[h]\eqbydef\Ih u$
and $\xi=\min_{T\in\Th}(\frac12,\tc\mu_0)$,
  \begin{multline}
    \label{eq:conv.rate}
    \gamma'\xi\norm[\sharp,h]{\shu[h]-\su[h]}\\
	\le \left\{\sum_{T\in\Th}\left[
    (\diff[T]\norm[H^{k+2}(T)]{u}^2+\tc^{-1}\norm[H^{k+1}(T)]{u}^2)h_T^{2(k+1)}
    + \velc \min(1,\Pe_T) h_T^{2k+1}\norm[H^{k+1}(T)]{u}^2
    \right]
    \right\}^{\nicefrac12}
  \end{multline}
where $\Pe_{T}=\max_{F\in\Fh[T]}\norm[L^\infty(F)]{\Pe_{TF}}$ is a local P\'eclet number 
(conventionally, $\norm[L^\infty(F)]{\Pe_{TF}}=+\infty$ if $\diff[F]=0$).
\end{theorem}

\begin{remark}[Regime-dependent estimate]\label{rem:regimes}
Using the local P\'eclet number in
\eqref{eq:conv.rate} allows us to establish an error estimate which
locally adjusts to the various regimes of \eqref{eq:strong}.
In mesh cells where diffusion dominates so that $\Pe_T\le h_T$, the contribution to the right-hand side of \eqref{eq:strong}
is $\mathcal O(h_T^{2(k+1)})$. In mesh cells where advection dominates so that $\Pe_T\ge 1$, the contribution is $\mathcal O(h_T^{2k+1})$. The transition region,
where $\Pe_T$ is between $h_T$ and $1$, corresponds to intermediate orders of convergence. Notice also that the diffusive contribution exhibits the superconvergent behavior $\mathcal O(h_T^{2(k+1)})$ typical of HHO methods, see~\cite{Di-Pietro.Ern.ea:14,Di-Pietro.Ern:14a}. As a result, the balancing with the advective contribution is slightly different with respect to other methods where the diffusive contribution typically scales as $\mathcal O(h_T^{2k})$.
\end{remark}

\subsection{Link with Hybrid Mixed Mimetic methods}\label{sec:link}

We assume here that the diffusion is not degenerate, i.e. $\ldiff>0$, and show that, under a slight modification of the definition of $\Pe_{TF}$, see~\eqref{def:newPe} below, the present discontinuous-skeletal method for $k=0$ corresponds to a face-based Hybrid Mimetic Mixed (HMM) method
studied for advective-diffusive equations in~\cite{Beirao-da-Veiga.Droniou.ea:10}. In this section, we consider that the local P\'eclet number $\Pe_{TF}$ is no
longer a function defined on the edge $F$, but the average of this function, i.e.,
\begin{equation}\label{def:newPe}
\Pe_{TF}=\frac{1}{|F|}\int_F \frac{h_F}{\diff[F]}\vel\SCAL\normal_{TF}.
\end{equation}
With this new definition, and assuming that $\ldiff>0$,
the face-based HMM method for \eqref{eq:strong} with $\reac=0$ can be written
(see \cite[Eqs.~(2.48)--(2.49)]{Beirao-da-Veiga.Droniou.ea:10})
in the flux balance and continuity form as follows:
\begin{eqnarray}
	&&\forall T\in\Fh\,:\,\sum_{F\in\Fh[T]}|F|\left[(\mathbb{F}_{\rm d})_{TF}
	+(\mathbb{F}_{\rm a})_{TF}\right] = \int_T f\,\label{HMM:eq1}\\
	&&\forall F\in \Fh[T]\cap\Fh[T']\mbox{ with $T\neq T'$}\,:\;(\mathbb{F}_{\rm d})_{TF}+
	(\mathbb{F}_{\rm a})_{TF}+(\mathbb{F}_{\rm d})_{T'F}+(\mathbb{F}_{\rm a})_{T'F}=0,
	\label{HMM:eq2}
\end{eqnarray}
where $\mathbb{F}_{\rm d}$ and $\mathbb{F}_{\rm a}$ are diffusion and advection fluxes,
constructed from the unknown $\su[h]\in \Uhz$. 
We additionally assume that boundary conditions are strongly enforced by considering the space $\Uhz[,0]\eqbydef\left\{\sv[h]\in\Uhz\st\unv[F]\equiv 0\quad\forall F\in\Fhb\right\}$ (we are entitled to strongly enforce boundary conditions since we assume $\ldiff>0$ in this section).
Taking $\sv[h]\in \Uhz[,0]$, multiplying \eqref{HMM:eq1} by the constant value $\unv[T]$, summing on the cells $T\in \Th$, and using the flux conservativity \eqref{HMM:eq2} and the strong boundary condition to introduce the constant value
$\unv[F]$ in the sums,
we see that these two equations are equivalent to
\begin{equation}\label{eq:link1}
	\sum_{T\in\Fh}\sum_{F\in\Fh[T]}|F|\left[(\mathbb{F}_{\rm d})_{TF}+(\mathbb{F}_{\rm a})_{TF}\right]
	(\unv[T]-\unv[F])=l_h(\sv[h]),
\end{equation}
for all $\sv[h]\in \Uhz[,0]$.
As seen in \cite{Di-Pietro.Ern.ea:14}, the definition of the diffusive flux
$\mathbb{F}_{\rm d}$ in \cite[Eq. (2.25)]{Droniou.Eymard.ea:10} shows that,
when the stabilization matrices $\mathbb{B}_T$ in the HMM method
are diagonal with coefficients $(\frac{\diff[T]}{h_F}|F|)_{F\in\Fh[T]}$,
the local diffusive term $\sum_{F\in\Fh[T]}|F|(\mathbb{F}_{\rm d})_{TF}(\unv[T]-\unv[F])$
is identical to the local diffusive bilinear form $a_{\diff,T}$ defined in~\eqref{eq:aT.nu}.
Therefore, it remains to study the advective term in \eqref{eq:link1} and
see that it corresponds to $a_{\vel,0,h}(\su[h],\sv[h])$.
With the choice \eqref{def:newPe}, using the diffusive scaling mentioned in
\cite[\S 2.4.1]{Beirao-da-Veiga.Droniou.ea:10} and applying a local geometric scaling
based on the edge diameter $h_F$ rather than the distance between the
two neighboring cell centers, the advective flux is written \cite[Eq. (2.46)]{Beirao-da-Veiga.Droniou.ea:10}
\[
	(\mathbb{F}_{\rm a})_{TF}=\tfrac{\diff[F]}{h_F}\left(\Ap(\Pe_{TF})\unu[T]
-\Am(\Pe_{TF})\unu[F]\right).
\]
Since $\Ap(s)-\Am(s)=s$ and invoking the assumption $\DIV\vel\equiv 0$,
we find that the advective contribution in \eqref{eq:link1} is
\begin{align*}
	&\sum_{T\in\Fh}\sum_{F\in\Fh[T]}|F|\left[\tfrac{\diff[F]}{h_F}\Ap(\Pe_{TF})\unu[T]
-\tfrac{\diff[F]}{h_F}\Am(\Pe_{TF})\unu[F]\right](\unv[T]-\unv[F])\\
	&=\sum_{T\in\Fh}\sum_{F\in\Fh[T]}|F|\left[\tfrac{\diff[F]}{h_F}
\left(\Ap(\Pe_{TF})-\Am(\Pe_{TF})\right)\unu[T]
+\tfrac{\diff[F]}{h_F}\Am(\Pe_{TF})(\unu[T]-\unu[F])\right](\unv[T]-\unv[F])\\
	&=\sum_{T\in\Fh}\sum_{F\in\Fh[T]}\left(\int_F \vel\SCAL\normal_{TF}\right)\unu[T]
(\unv[T]-\unv[F])+\sum_{T\in\Fh}\sum_{F\in\Fh[T]}|F|
\tfrac{\diff[F]}{h_F}\Am(\Pe_{TF})(\unu[T]-\unu[F])(\unv[T]-\unv[F])\\
	&=\sum_{T\in\Fh}\unu[T]\unv[T]\int_T \DIV\vel
-\sum_{T\in\Fh}\unu[T]\sum_{F\in\Fh[T]}\left(\int_F \vel\SCAL\normal_{TF}\right)\unv[F]
+s_{\vel,h}^-(\su[h],\sv[h])\\
	&=-\sum_{T\in\Fh}|T|\unu[T]\left(\frac{1}{|T|}\sum_{F\in\Fh[T]}\left(\int_F \vel\SCAL\normal_{TF}\right)\unv[F]\right)
+s_{\vel,h}^-(\su[h],\sv[h]).
\end{align*}
It is then a simple matter of inspecting the definition \eqref{eq:GBT}
in the case $k=0$ to notice that
\[
\GBT\sv=\frac{1}{|T|}\sum_{F\in\Fh[T]}\left(\int_F \vel\SCAL\normal_{TF}\right) \unv[F],
\]
and therefore conclude that the advective contribution in \eqref{eq:link1}
is indeed $a_{\vel,0,h}(\su[h],\sv[h])$.


\subsection{Numerical results}
\label{sec:num.val}

To close this section, we provide numerical results illustrating the error estimate of Theorem~\ref{thm:err.est}.

\subsubsection{Uniform diffusion}
\label{sec:num.val:test1}

To numerically assess the sharpness of estimate~\eqref{eq:conv.rate} in the uniform diffusion case, we solve on the unit square the problem~\eqref{eq:strong} with boundary conditions and right-hand side inferred from the following exact solution:
$$
u(\vec{x}) = \sin(\pi x_1)\sin(\pi x_2).
$$
We take $\vel(\vec{x})=(\nicefrac12-x_2,x_1-\nicefrac12)$, $\mu=1$, and we let $\nu$ vary in $\{0,10^{-3},1\}$.
In Figure~\ref{fix:num.val:test1} we display the convergence results for the three mesh families depicted in Figure~\ref{fig:meshes:test1}.
From top to bottom, these correspond, respectively, to the mesh families 1 (triangular) and 4.1 (Kershaw) of the FVCA5 benchmark~\cite{Herbin.Hubert:08}, and to the (predominantly) hexagonal mesh family first introduced in~\cite{Di-Pietro.Lemaire:14}.
Each line in Figure~\ref{fix:num.val:test1} corresponds to a different mesh family, and the value of $\diff$ increases from left to right.
In all of the cases, an increase in the asymptotic convergence rate of about half a unit is observed as we increase the value of $\diff$, as predicted by~\eqref{eq:conv.rate}.
The results also show that the method behaves consistently on a variety of meshes including general polygonal elements.
The slightly higher convergence rates for the Kershaw mesh family are possibly due to the fact that the mesh regularity changes when refining.

\begin{figure}
\ifDdP
  \centering
  \includegraphics[height=3.5cm]{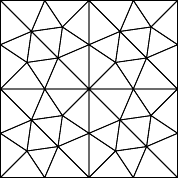}
  \hspace{1ex}
  \includegraphics[height=3.5cm]{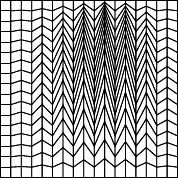}
  \hspace{1ex}
  \includegraphics[height=3.5cm]{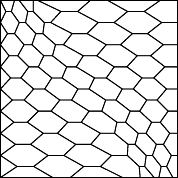}
\fi
  \caption{Meshes for the test case of Section~\ref{sec:num.val:test1}. From left to right, the meshes are refereed to as triangular, Kershaw and hexagonal, respectively\label{fig:meshes:test1}}
\end{figure}

\begin{figure}
\ifDdP
  \begin{minipage}[b]{0.30\linewidth}
    \begin{tikzpicture}[scale=0.60]
      \begin{loglogaxis}[
          xmin=0.001,
          legend style = {
            legend pos = south east
          }
        ]
        \addplot table[x=meshsize,y={create col/linear regression={y=err_sharp}}] {dat/adho/adho_0_0_mesh1.dat}
        coordinate [pos=0.75] (A)
        coordinate [pos=1.00] (B);
        \xdef\slopea{\pgfplotstableregressiona}
        \draw (A) -| (B) node[pos=0.75,anchor=east] {\pgfmathprintnumber{\slopea}};
        \addplot table[x=meshsize,y={create col/linear regression={y=err_sharp}}] {dat/adho/adho_0_1_mesh1.dat}
        coordinate [pos=0.75] (A)
        coordinate [pos=1.00] (B);
        \xdef\slopeb{\pgfplotstableregressiona}
        \draw (A) -| (B) node[pos=0.75,anchor=east] {\pgfmathprintnumber{\slopeb}};
        \addplot table[x=meshsize,y={create col/linear regression={y=err_sharp}}] {dat/adho/adho_0_2_mesh1.dat}
        coordinate [pos=0.75] (A)
        coordinate [pos=1.00] (B);
        \xdef\slopec{\pgfplotstableregressiona}
        \draw (A) -| (B) node[pos=0.75,anchor=east] {\pgfmathprintnumber{\slopec}};
        \addplot table[x=meshsize,y={create col/linear regression={y=err_sharp}}] {dat/adho/adho_0_3_mesh1.dat}
        coordinate [pos=0.75] (A)
        coordinate [pos=1.00] (B);
        \xdef\sloped{\pgfplotstableregressiona}
        \draw (A) -| (B) node[pos=0.75,anchor=east] {\pgfmathprintnumber{\sloped}};
        \legend{$k=0$,$k=1$,$k=2$,$k=3$};
      \end{loglogaxis}
    \end{tikzpicture}
    \subcaption{$\diff=0$, triangular}
  \end{minipage}
  \begin{minipage}[b]{0.30\linewidth}
    \begin{tikzpicture}[scale=0.60]
      \begin{loglogaxis}[
          xmin=0.001,
          legend style = {
            legend pos = south east
          }
        ]
        \addplot table[x=meshsize,y={create col/linear regression={y=err_sharp}}] {dat/adho/adho_0.001_0_mesh1.dat}
        coordinate [pos=0.75] (A)
        coordinate [pos=1.00] (B);
        \xdef\slopea{\pgfplotstableregressiona}
        \draw (A) -| (B) node[pos=0.75,anchor=east] {\pgfmathprintnumber{\slopea}};
        \addplot table[x=meshsize,y={create col/linear regression={y=err_sharp}}] {dat/adho/adho_0.001_1_mesh1.dat}
        coordinate [pos=0.75] (A)
        coordinate [pos=1.00] (B);
        \xdef\slopeb{\pgfplotstableregressiona}
        \draw (A) -| (B) node[pos=0.75,anchor=east] {\pgfmathprintnumber{\slopeb}};
        \addplot table[x=meshsize,y={create col/linear regression={y=err_sharp}}] {dat/adho/adho_0.001_2_mesh1.dat}
        coordinate [pos=0.75] (A)
        coordinate [pos=1.00] (B);
        \xdef\slopec{\pgfplotstableregressiona}
        \draw (A) -| (B) node[pos=0.75,anchor=east] {\pgfmathprintnumber{\slopec}};
        \addplot table[x=meshsize,y={create col/linear regression={y=err_sharp}}] {dat/adho/adho_0.001_3_mesh1.dat}
        coordinate [pos=0.75] (A)
        coordinate [pos=1.00] (B);
        \xdef\sloped{\pgfplotstableregressiona}
        \draw (A) -| (B) node[pos=0.75,anchor=east] {\pgfmathprintnumber{\sloped}};
        \legend{$k=0$,$k=1$,$k=2$,$k=3$};
      \end{loglogaxis}
    \end{tikzpicture}
    \subcaption{$\diff=10^{-3}$, triangular}
  \end{minipage}
  \begin{minipage}[b]{0.30\linewidth}
    \begin{tikzpicture}[scale=0.60]
      \begin{loglogaxis}[
          xmin=0.001,
          legend style = {
            legend pos = south east
          }
        ]
        \addplot table[x=meshsize,y={create col/linear regression={y=err_sharp}}] {dat/adho/adho_1_0_mesh1.dat}
        coordinate [pos=0.75] (A)
        coordinate [pos=1.00] (B);
        \xdef\slopea{\pgfplotstableregressiona}
        \draw (A) -| (B) node[pos=0.75,anchor=east] {\pgfmathprintnumber{\slopea}};
        \addplot table[x=meshsize,y={create col/linear regression={y=err_sharp}}] {dat/adho/adho_1_1_mesh1.dat}
        coordinate [pos=0.75] (A)
        coordinate [pos=1.00] (B);
        \xdef\slopeb{\pgfplotstableregressiona}
        \draw (A) -| (B) node[pos=0.75,anchor=east] {\pgfmathprintnumber{\slopeb}};
        \addplot table[x=meshsize,y={create col/linear regression={y=err_sharp}}] {dat/adho/adho_1_2_mesh1.dat}
        coordinate [pos=0.75] (A)
        coordinate [pos=1.00] (B);
        \xdef\slopec{\pgfplotstableregressiona}
        \draw (A) -| (B) node[pos=0.75,anchor=east] {\pgfmathprintnumber{\slopec}};
        \addplot table[x=meshsize,y={create col/linear regression={y=err_sharp}}] {dat/adho/adho_1_3_mesh1.dat}
        coordinate [pos=0.75] (A)
        coordinate [pos=1.00] (B);
        \xdef\sloped{\pgfplotstableregressiona}
        \draw (A) -| (B) node[pos=0.75,anchor=east] {\pgfmathprintnumber{\sloped}};
        \legend{$k=0$,$k=1$,$k=2$,$k=3$};
      \end{loglogaxis}
    \end{tikzpicture}
    \subcaption{$\diff=\pgfmathprintnumber{1}$, triangular}
  \end{minipage}
  \vspace{2ex}
  \\
  \begin{minipage}[b]{0.30\linewidth}
    \begin{tikzpicture}[scale=0.60]
      \begin{loglogaxis}[
          xmin = 0.0025,
          legend style = {
            legend pos = south east
          }
        ]
        \addplot table[x=meshsize,y={create col/linear regression={y=err_sharp}}] {dat/adho/adho_0_0_mesh4_1.dat}
        coordinate [pos=0.75] (A)
        coordinate [pos=1.00] (B);
        \xdef\slopea{\pgfplotstableregressiona}
        \draw (A) -| (B) node[pos=0.75,anchor=east] {\pgfmathprintnumber{\slopea}};
        \addplot table[x=meshsize,y={create col/linear regression={y=err_sharp}}] {dat/adho/adho_0_1_mesh4_1.dat}
        coordinate [pos=0.75] (A)
        coordinate [pos=1.00] (B);
        \xdef\slopeb{\pgfplotstableregressiona}
        \draw (A) -| (B) node[pos=0.75,anchor=east] {\pgfmathprintnumber{\slopeb}};
        \addplot table[x=meshsize,y={create col/linear regression={y=err_sharp}}] {dat/adho/adho_0_2_mesh4_1.dat}
        coordinate [pos=0.75] (A)
        coordinate [pos=1.00] (B);
        \xdef\slopec{\pgfplotstableregressiona}
        \draw (A) -| (B) node[pos=0.75,anchor=east] {\pgfmathprintnumber{\slopec}};
        \addplot table[x=meshsize,y={create col/linear regression={y=err_sharp}}] {dat/adho/adho_0_3_mesh4_1.dat}
        coordinate [pos=0.75] (A)
        coordinate [pos=1.00] (B);
        \xdef\sloped{\pgfplotstableregressiona}
        \draw (A) -| (B) node[pos=0.75,anchor=east] {\pgfmathprintnumber{\sloped}};
        \legend{$k=0$,$k=1$,$k=2$,$k=3$};
      \end{loglogaxis}
    \end{tikzpicture}
    \subcaption{$\diff=0$, Kershaw}
  \end{minipage}
  \begin{minipage}[b]{0.30\linewidth}
    \begin{tikzpicture}[scale=0.60]
      \begin{loglogaxis}[
          xmin = 0.0025,
          legend style = {
            legend pos = south east
          }
        ]
        \addplot table[x=meshsize,y={create col/linear regression={y=err_sharp}}] {dat/adho/adho_0.001_0_mesh4_1.dat}
        coordinate [pos=0.75] (A)
        coordinate [pos=1.00] (B);
        \xdef\slopea{\pgfplotstableregressiona}
        \draw (A) -| (B) node[pos=0.75,anchor=east] {\pgfmathprintnumber{\slopea}};
        \addplot table[x=meshsize,y={create col/linear regression={y=err_sharp}}] {dat/adho/adho_0.001_1_mesh4_1.dat}
        coordinate [pos=0.75] (A)
        coordinate [pos=1.00] (B);
        \xdef\slopeb{\pgfplotstableregressiona}
        \draw (A) -| (B) node[pos=0.75,anchor=east] {\pgfmathprintnumber{\slopeb}};
        \addplot table[x=meshsize,y={create col/linear regression={y=err_sharp}}] {dat/adho/adho_0.001_2_mesh4_1.dat}
        coordinate [pos=0.75] (A)
        coordinate [pos=1.00] (B);
        \xdef\slopec{\pgfplotstableregressiona}
        \draw (A) -| (B) node[pos=0.75,anchor=east] {\pgfmathprintnumber{\slopec}};
        \addplot table[x=meshsize,y={create col/linear regression={y=err_sharp}}] {dat/adho/adho_0.001_3_mesh4_1.dat}
        coordinate [pos=0.75] (A)
        coordinate [pos=1.00] (B);
        \xdef\sloped{\pgfplotstableregressiona}
        \draw (A) -| (B) node[pos=0.75,anchor=east] {\pgfmathprintnumber{\sloped}};
        \legend{$k=0$,$k=1$,$k=2$,$k=3$};
      \end{loglogaxis}
    \end{tikzpicture}
    \subcaption{$\diff=10^{-3}$, Kershaw}
  \end{minipage}
  \begin{minipage}[b]{0.30\linewidth}
    \begin{tikzpicture}[scale=0.60]
      \begin{loglogaxis}[
          xmin = 0.0025,
          legend style = {
            legend pos = south east
          }
        ]
        \addplot table[x=meshsize,y={create col/linear regression={y=err_sharp}}] {dat/adho/adho_1_0_mesh4_1.dat}
        coordinate [pos=0.75] (A)
        coordinate [pos=1.00] (B);
        \xdef\slopea{\pgfplotstableregressiona}
        \draw (A) -| (B) node[pos=0.75,anchor=east] {\pgfmathprintnumber{\slopea}};
        \addplot table[x=meshsize,y={create col/linear regression={y=err_sharp}}] {dat/adho/adho_1_1_mesh4_1.dat}
        coordinate [pos=0.75] (A)
        coordinate [pos=1.00] (B);
        \xdef\slopeb{\pgfplotstableregressiona}
        \draw (A) -| (B) node[pos=0.75,anchor=east] {\pgfmathprintnumber{\slopeb}};
        \addplot table[x=meshsize,y={create col/linear regression={y=err_sharp}}] {dat/adho/adho_1_2_mesh4_1.dat}
        coordinate [pos=0.75] (A)
        coordinate [pos=1.00] (B);
        \xdef\slopec{\pgfplotstableregressiona}
        \draw (A) -| (B) node[pos=0.75,anchor=east] {\pgfmathprintnumber{\slopec}};
        \addplot table[x=meshsize,y={create col/linear regression={y=err_sharp}}] {dat/adho/adho_1_3_mesh4_1.dat}
        coordinate [pos=0.75] (A)
        coordinate [pos=1.00] (B);
        \xdef\sloped{\pgfplotstableregressiona}
        \draw (A) -| (B) node[pos=0.75,anchor=east] {\pgfmathprintnumber{\sloped}};
        \legend{$k=0$,$k=1$,$k=2$,$k=3$};
      \end{loglogaxis}
    \end{tikzpicture}
    \subcaption{$\diff=\pgfmathprintnumber{1}$, Kershaw}
  \end{minipage}
  \vspace{2ex}
  \\ 
  \begin{minipage}[b]{0.30\linewidth}
    \begin{tikzpicture}[scale=0.60]
      \begin{loglogaxis}[
          xmin=0.0025,
          legend style = {
            legend pos = south east
          }
        ]
        \addplot table[x=meshsize,y={create col/linear regression={y=err_sharp}}] {dat/adho/adho_0_0_pi6_tiltedhexagonal.dat}
        coordinate [pos=0.75] (A)
        coordinate [pos=1.00] (B);
        \xdef\slopea{\pgfplotstableregressiona}
        \draw (A) -| (B) node[pos=0.75,anchor=east] {\pgfmathprintnumber{\slopea}};
        \addplot table[x=meshsize,y={create col/linear regression={y=err_sharp}}] {dat/adho/adho_0_1_pi6_tiltedhexagonal.dat}
        coordinate [pos=0.75] (A)
        coordinate [pos=1.00] (B);
        \xdef\slopeb{\pgfplotstableregressiona}
        \draw (A) -| (B) node[pos=0.75,anchor=east] {\pgfmathprintnumber{\slopeb}};
        \addplot table[x=meshsize,y={create col/linear regression={y=err_sharp}}] {dat/adho/adho_0_2_pi6_tiltedhexagonal.dat}
        coordinate [pos=0.75] (A)
        coordinate [pos=1.00] (B);
        \xdef\slopec{\pgfplotstableregressiona}
        \draw (A) -| (B) node[pos=0.75,anchor=east] {\pgfmathprintnumber{\slopec}};
        \addplot table[x=meshsize,y={create col/linear regression={y=err_sharp}}] {dat/adho/adho_0_3_pi6_tiltedhexagonal.dat}
        coordinate [pos=0.75] (A)
        coordinate [pos=1.00] (B);
        \xdef\sloped{\pgfplotstableregressiona}
        \draw (A) -| (B) node[pos=0.75,anchor=east] {\pgfmathprintnumber{\sloped}};
        \legend{$k=0$,$k=1$,$k=2$,$k=3$};
      \end{loglogaxis}
    \end{tikzpicture}
    \subcaption{$\diff=0$, hexagonal}
  \end{minipage}
  \begin{minipage}[b]{0.30\linewidth}
    \begin{tikzpicture}[scale=0.60]
      \begin{loglogaxis}[
          xmin=0.0025,
          legend style = {
            legend pos = south east
          }
        ]
        \addplot table[x=meshsize,y={create col/linear regression={y=err_sharp}}] {dat/adho/adho_0.001_0_pi6_tiltedhexagonal.dat}
        coordinate [pos=0.75] (A)
        coordinate [pos=1.00] (B);
        \xdef\slopea{\pgfplotstableregressiona}

        \draw (A) -| (B) node[pos=0.75,anchor=east] {\pgfmathprintnumber{\slopea}};
        \addplot table[x=meshsize,y={create col/linear regression={y=err_sharp}}] {dat/adho/adho_0.001_1_pi6_tiltedhexagonal.dat}
        coordinate [pos=0.75] (A)
        coordinate [pos=1.00] (B);
        \xdef\slopeb{\pgfplotstableregressiona}
        \draw (A) -| (B) node[pos=0.75,anchor=east] {\pgfmathprintnumber{\slopeb}};
        \addplot table[x=meshsize,y={create col/linear regression={y=err_sharp}}] {dat/adho/adho_0.001_2_pi6_tiltedhexagonal.dat}
        coordinate [pos=0.75] (A)
        coordinate [pos=1.00] (B);
        \xdef\slopec{\pgfplotstableregressiona}
        \draw (A) -| (B) node[pos=0.75,anchor=east] {\pgfmathprintnumber{\slopec}};
        \addplot table[x=meshsize,y={create col/linear regression={y=err_sharp}}] {dat/adho/adho_0.001_3_pi6_tiltedhexagonal.dat}
        coordinate [pos=0.75] (A)
        coordinate [pos=1.00] (B);
        \xdef\sloped{\pgfplotstableregressiona}
        \draw (A) -| (B) node[pos=0.75,anchor=east] {\pgfmathprintnumber{\sloped}};
        \legend{$k=0$,$k=1$,$k=2$,$k=3$};
      \end{loglogaxis}
    \end{tikzpicture}
    \subcaption{$\diff=10^{-3}$, hexagonal}
  \end{minipage}
  \begin{minipage}[b]{0.30\linewidth}
    \begin{tikzpicture}[scale=0.60]
      \begin{loglogaxis}[
          xmin=0.0025,
          legend style = {
            legend pos = south east
          }
        ]
        \addplot table[x=meshsize,y={create col/linear regression={y=err_sharp}}] {dat/adho/adho_1_0_pi6_tiltedhexagonal.dat}
        coordinate [pos=0.75] (A)
        coordinate [pos=1.00] (B);
        \xdef\slopea{\pgfplotstableregressiona}
        \draw (A) -| (B) node[pos=0.75,anchor=east] {\pgfmathprintnumber{\slopea}};
        \addplot table[x=meshsize,y={create col/linear regression={y=err_sharp}}] {dat/adho/adho_1_1_pi6_tiltedhexagonal.dat}
        coordinate [pos=0.75] (A)
        coordinate [pos=1.00] (B);
        \xdef\slopeb{\pgfplotstableregressiona}
        \draw (A) -| (B) node[pos=0.75,anchor=east] {\pgfmathprintnumber{\slopeb}};
        \addplot table[x=meshsize,y={create col/linear regression={y=err_sharp}}] {dat/adho/adho_1_2_pi6_tiltedhexagonal.dat}
        coordinate [pos=0.75] (A)
        coordinate [pos=1.00] (B);
        \xdef\slopec{\pgfplotstableregressiona}
        \draw (A) -| (B) node[pos=0.75,anchor=east] {\pgfmathprintnumber{\slopec}};
        \addplot table[x=meshsize,y={create col/linear regression={y=err_sharp}}] {dat/adho/adho_1_3_pi6_tiltedhexagonal.dat}
        coordinate [pos=0.75] (A)
        coordinate [pos=1.00] (B);
        \xdef\sloped{\pgfplotstableregressiona}
        \draw (A) -| (B) node[pos=0.75,anchor=east] {\pgfmathprintnumber{\sloped}};
        \legend{$k=0$,$k=1$,$k=2$,$k=3$};
      \end{loglogaxis}
    \end{tikzpicture}
    \subcaption{$\diff=\pgfmathprintnumber{1}$, hexagonal}
  \end{minipage}
\fi
  \caption{$\norm[\sharp,h]{\shu[h]-\su[h]}$-norm vs. $h$ for different mesh families (rows) and values of the diffusion coefficient $\diff$ (columns) in the test case of Section~\ref{sec:num.val:test1}\label{fix:num.val:test1}}
\end{figure}
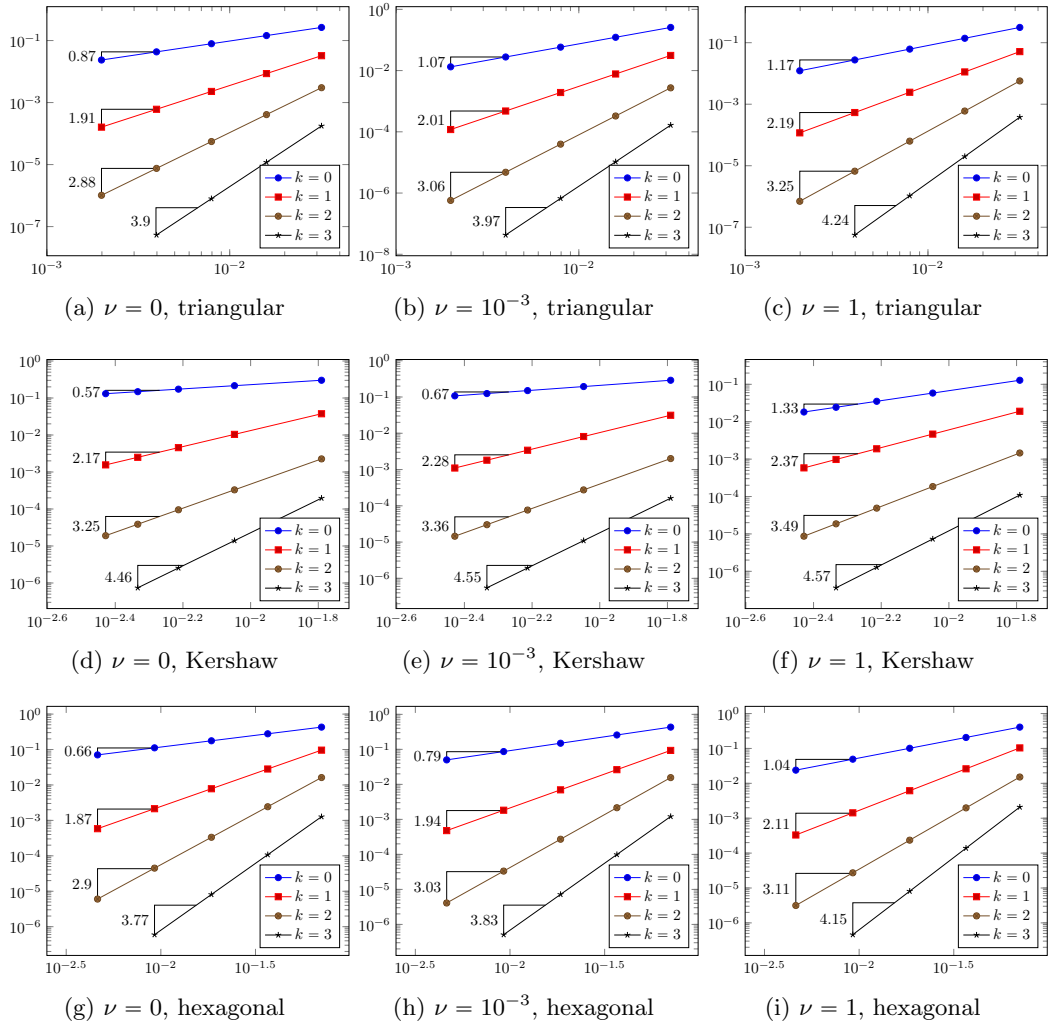

\subsubsection{Locally degenerate diffusion}
\label{sec:num.val:test2}

\begin{figure}
  \centering
\ifDdP
  \includegraphics[]{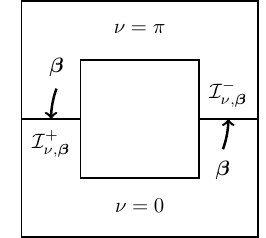} 
  \hspace{2ex}
  \includegraphics[height=4cm]{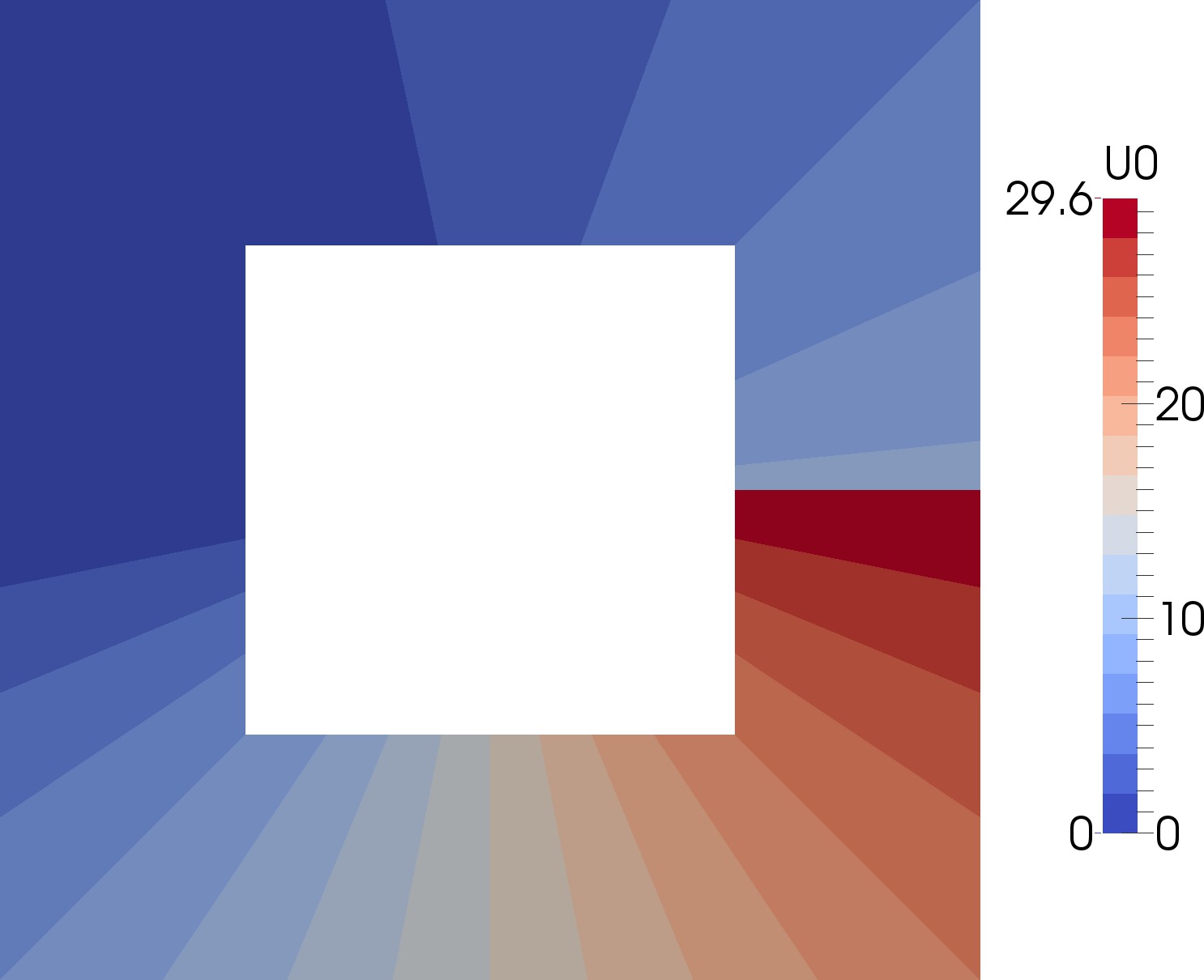}
\fi
  \caption{Configuration for the test case of Section~\ref{sec:num.val:test2} (left) and numerical solution for $k=3$ and $h=1.29{\times}10^{-2}$ (right). The jump discontinuity across $\Int[-]$ is clearly visible.\label{fig:dpeg}}
\end{figure}

To validate the method in the locally degenerate case, we consider the configuration originally proposed in~\cite[Section~6.1]{Di-Pietro.Ern.ea:08}, cf. Figure~\ref{fig:dpeg}.
The domain is $\Omega = (-1,1)^2\setminus [-0.5,0.5]^2$. Denoting by $(r,\theta)$ the standard polar coordinates (with azimuth $\theta$ measured counterclockwise starting from the positive $x$-axis) and by $\vec{e}_\theta$ the azimuthal vector, the problem coefficients are
$$
\diff(\theta,r) = \begin{cases} \pi & \text{if $0<\theta<\pi$,} \\ 0 & \text{if $\pi<\theta<2\pi$,}\end{cases}\quad
\vel(\theta,r) = \frac{\vec{e}_\theta}{r},\quad
\reac = 10^{-6},\quad
$$
The exact solution, also used to infer the value of the forcing term $f$ and boundary datum $g$, is given by
$$
u(\theta,r) = \begin{cases} (\theta-\pi)^2 & \text{if $0<\theta<\pi$,} \\ 
  3\pi(\theta-\pi) & \text{if $\pi<\theta<2\pi$.}\end{cases}
$$
In Figure~\ref{fig:num.val:test2} we show the convergence results for a refined family of triangular meshes.
The left panel displays the $L^2$-error on the potential measured by the quantity $\left\{\sum_{T\in\Th}\norm[T]{\unhu[T]-\unu[T]}^2\right\}^{\nicefrac12}$ with $\shu\eqbydef\Ih u$, while the right panel 
contains the error in the $\norm[\sharp,h]{{\cdot}}$-norm defined by~\eqref{eq:norm.h}.
In both cases the relative error is displayed and we have taken $\varsigma=1$.

\begin{figure}
 \centering
\ifDdP
 \begin{minipage}[b]{0.45\linewidth}
   \begin{tikzpicture}[scale=0.85]
     \begin{loglogaxis}[
         xmin = 1.5e-3, 
         ymax = 0,
         ymin = 1e-10,
         legend style = {
           legend pos = south east
         }
       ]
       \addplot table[x=meshsize,y={create col/linear regression={y=err_uh}}] {dat/dpeg/dpeg_0.dat}
       coordinate [pos=0.75] (A)
       coordinate [pos=1.00] (B);
       \xdef\slopea{\pgfplotstableregressiona}
       \draw (A) -| (B) node[pos=0.75,anchor=east] {\pgfmathprintnumber{\slopea}};
       \addplot table[x=meshsize,y={create col/linear regression={y=err_uh}}] {dat/dpeg/dpeg_1.dat}
       coordinate [pos=0.75] (A)
       coordinate [pos=1.00] (B);
       \xdef\slopeb{\pgfplotstableregressiona}
       \draw (A) -| (B) node[pos=0.75,anchor=east] {\pgfmathprintnumber{\slopeb}};
       \addplot table[x=meshsize,y={create col/linear regression={y=err_uh}}] {dat/dpeg/dpeg_2.dat}
       coordinate [pos=0.75] (A)
       coordinate [pos=1.00] (B);
       \xdef\slopec{\pgfplotstableregressiona}
       \draw (A) -| (B) node[pos=0.75,anchor=east] {\pgfmathprintnumber{\slopec}};
       \addplot table[x=meshsize,y={create col/linear regression={y=err_uh}}] {dat/dpeg/dpeg_3.dat}
       coordinate [pos=0.75] (A)
       coordinate [pos=1.00] (B);
       \xdef\sloped{\pgfplotstableregressiona}
       \draw (A) -| (B) node[pos=1.00,anchor=east] {\pgfmathprintnumber{\sloped}};
       \legend{$k=0$,$k=1$,$k=2$,$k=3$};
     \end{loglogaxis}
   \end{tikzpicture}
   \subcaption{$L^2$-error on $u$}
 \end{minipage}
 \begin{minipage}[b]{0.45\linewidth}
   \begin{tikzpicture}[scale=0.85]
     \begin{loglogaxis}[
         xmin = 1.5e-3, 
         ymax = 0,
         ymin = 1e-10,
         legend style = {
           legend pos = south east
         }
       ]
       \addplot table[x=meshsize,y={create col/linear regression={y=err_sharp}}] {dat/dpeg/dpeg_0.dat}
       coordinate [pos=0.75] (A)
       coordinate [pos=1.00] (B);
       \xdef\slopea{\pgfplotstableregressiona}
       \draw (A) -| (B) node[pos=0.75,anchor=east] {\pgfmathprintnumber{\slopea}};
       \addplot table[x=meshsize,y={create col/linear regression={y=err_sharp}}] {dat/dpeg/dpeg_1.dat}
       coordinate [pos=0.75] (A)
       coordinate [pos=1.00] (B);
       \xdef\slopeb{\pgfplotstableregressiona}
       \draw (A) -| (B) node[pos=0.75,anchor=east] {\pgfmathprintnumber{\slopeb}};
       \addplot table[x=meshsize,y={create col/linear regression={y=err_sharp}}] {dat/dpeg/dpeg_2.dat}
       coordinate [pos=0.75] (A)
       coordinate [pos=1.00] (B);
       \xdef\slopec{\pgfplotstableregressiona}
       \draw (A) -| (B) node[pos=0.75,anchor=east] {\pgfmathprintnumber{\slopec}};
       \addplot table[x=meshsize,y={create col/linear regression={y=err_sharp}}] {dat/dpeg/dpeg_3.dat}
       coordinate [pos=0.75] (A)
       coordinate [pos=1.00] (B);
       \xdef\sloped{\pgfplotstableregressiona}
       \draw (A) -| (B) node[pos=1.00,anchor=east] {\pgfmathprintnumber{\sloped}};
       \legend{$k=0$,$k=1$,$k=2$,$k=3$};
     \end{loglogaxis}
   \end{tikzpicture}
   \subcaption{$\norm[\sharp,h]{\shu[h]-\su[h]}$}
 \end{minipage}
\fi
 \caption{Convergence results for the locally degenerate test case of Section~\ref{sec:num.val:test2}\label{fig:num.val:test2}}
\end{figure}
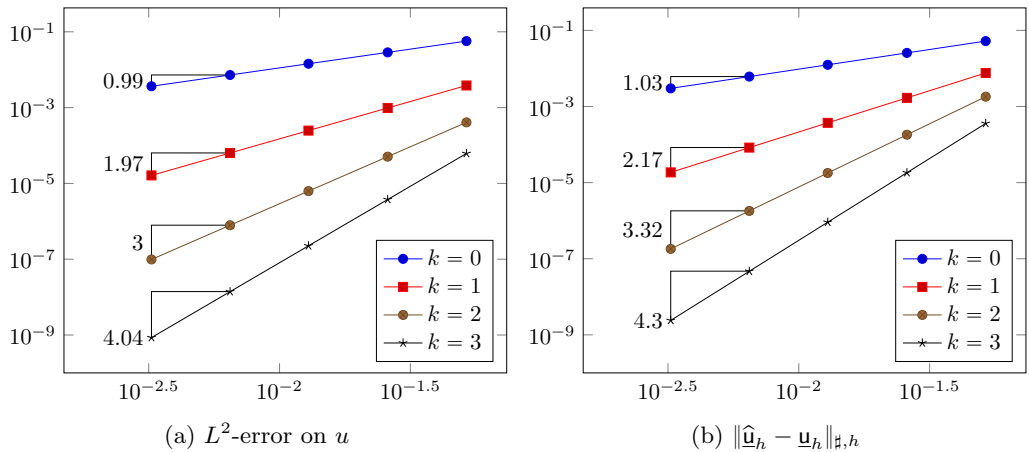

\section{Proofs}
\label{sec:analysis}

This section is concerned with the proof of our two main results: Lemma~\ref{lem:stab.ah} on stability and Theorem~\ref{thm:err.est} on the error estimate.
In what follows, we often abbreviate by $a\lesssim b$ the inequality $a\le C b$ with $C>0$ independent of $h$, $\diff$, $\vel$, and $\reac$, but possibly depending on $\varrho$, $d$, and $k$.

\subsection{Stability analysis}
\label{sec:proof_stab}
 
This section is organized as follows. First, we examine separately the coercivity of the diffusive and the advective-reactive bilinear forms. Combining these results readily yields the coercivity of the bilinear form $a_h$, see~\eqref{eq:coer.ah} in Lemma~\ref{lem:stab.ah}. Then, we prove the inf-sup condition~\eqref{eq:inf-sup.ah}.

\begin{lemma}[Stability of $a_{\diff,h}$]
  \label{lem:stab.cont.ah.nu}
Assume $\varsigma\ge1+\tfrac{C_{\rm tr}^2\Np}{2}$. Then, for all $\sv[h]\in\Uh$, the following holds:
  \begin{equation*}
    \label{eq:stab.cont.ah.nu}
    \frac12 \norm[\diff,h]{\sv[h]}^2 \lesssim a_{\diff,h}(\sv[h],\sv[h]).
  \end{equation*}
\end{lemma}

\begin{proof}
We use
the Cauchy--Schwarz inequality, the discrete trace~\eqref{eq:trace.disc} inequalities, the definition~\eqref{eq:norm.nu}
of the $\norm[\diff,T]{{\cdot}}$-seminorm, and we recall~\eqref{eq:Np} to obtain
\begin{equation*}
  \label{est:a.nu.bdr}
  \begin{aligned}
    \Bigg|\sum_{F\in\Fhb} (\diff[F]\GRAD\pTF\sv[T(F)]\SCAL\normal_{TF},\unv[F])_F\Bigg|
    &\le
    \left\{\sum_{F\in\Fhb} h_F\norm[F]{\diff[F]^{\nicefrac12}\GRAD\pTF\sv[T(F)]}^2\right\}^{\nicefrac12}
    \seminorm[\diff,\partial\Omega]{\sv[h]}
    \\
    &\le
    C_{\rm tr}\Np^{\nicefrac12}
    \left\{\sum_{T\in\Th}\norm[\nu,T]{\sv}^2\right\}^{\nicefrac12}
    \seminorm[\diff,\partial\Omega]{\sv[h]}.
  \end{aligned}
\end{equation*}
  Hence,
  \begin{equation}
    \label{eq:stab.cont.ah.nu:1}
    a_h(\sv[h],\sv[h])
    \ge 
    \sum_{T\in\Th}\norm[\diff,T]{\sv}^2
    - C_{\rm tr}\Np^{\nicefrac12}\left\{\sum_{T\in\Th}\norm[\nu,T]{\sv}^2\right\}^{\nicefrac12}
    \seminorm[\diff,\partial\Omega]{\sv[h]}
    + \varsigma\seminorm[\diff,\partial\Omega]{\sv[h]}^2.
  \end{equation}
  For a real number $A>0$, assuming $B>A^2$, the following inequality holds for all positive real numbers $x,y$:
  $x^2-2Axy+By^2\ge\frac{B-A^2}{1+B}(x^2+y^2)$.
  Using this result with $x=\left\{\sum_{T\in\Th}\norm[\diff,T]{\sv}^2\right\}^{\nicefrac12}$, $y=\seminorm[\diff,\partial\Omega]{\sv[h]}$, $2A=C_{\rm tr}\Np^{\nicefrac12}$, and $B=\varsigma$ in the right-hand side of~\eqref{eq:stab.cont.ah.nu:1} yields the assertion since $\frac{B-A^2}{1+B}\ge\frac12$.
\end{proof}

\begin{lemma}[Stability of $a_{\vel,\reac,h}$]
  \label{lem:stab:ah.vel.mu}
  Assume \ASSUM{A}{1}--\ASSUM{A}{3}. 
  The following holds for all $\sv[h]\in\Uh$:
  \begin{equation*}
    \label{eq:stab:ah.vel.mu}
    \eta \norm[\vel,\reac,h]{\sv[h]}^2
    \le a_{\vel,\reac,h}(\sv[h],\sv[h]),
  \end{equation*}
with $\eta\eqbydef\min_{T\in\Th}(1,\tc\reac_0)$.
\end{lemma}

\begin{proof}
\emph{Step 1}. Let us first prove that, for all $\sw[h],\sv[h]\in\Uh$, the following holds:
  \begin{multline}
    \label{eq:GBT.ibp}
    \sum_{T\in\Th}\Big\{
    (\GBT\sw,\unv[T])_T + (\unw[T],\GBT\sv)_T
    \Big\}
    \\
    =-\sum_{T\in\Th}\sum_{F\in\Fh}((\vel\SCAL\normal_{TF})(\unw[F]-\unw[T]),\unv[F]-\unv[T])_F
    +\sum_{F\in\Fhb}((\vel\SCAL\normal_F)\unw[F],\unv[F])_F.
  \end{multline}
  For all $T\in\Th$, we use~\eqref{eq:GBT.bis} with $\sv=\sw$ and $w=\unv[T]$ and~\eqref{eq:GBT} with $w=\unw[T]$ to infer that
  \begin{equation}
    \label{eq:GBT.ibp:1}
    \begin{aligned}
      &\sum_{T\in\Th}(\GBT\sw,\unv[T])_T
      \\
      &\quad=
      \sum_{T\in\Th}\left\{
      (\vel\SCAL\GRAD\unw[T],\unv[T])_T
      + \sum_{F\in\Fh[T]}((\vel\SCAL\normal_{TF})(\unw[F]-\unw[T]),\unv[T])_F
      \right\}
      \\
      &\quad=
      \sum_{T\in\Th}\left\{
      -(\unw[T],\GBT\sv)_T
      +\sum_{F\in\Fh[T]}((\vel\SCAL\normal_{TF})\unw[T],\unv[F])_F
      + \sum_{F\in\Fh[T]}((\vel\SCAL\normal_{TF})(\unw[F]-\unw[T]),\unv[T])_F
      \right\}.
    \end{aligned}
  \end{equation}
  Formula~\eqref{eq:GBT.ibp} follows adding to the right-hand side of~\eqref{eq:GBT.ibp:1} the quantity
  \begin{equation}
    \label{eq:GBT.ibp:2}
    0 =\sum_{F\in\Fhb}((\vel\SCAL\normal_F)\unw[F],\unv[F])_F
    -\sum_{T\in\Th}\sum_{F\in\Fh[T]}((\vel\SCAL\normal_{TF})\unw[F],\unv[F])_F.
  \end{equation}
  To prove~\eqref{eq:GBT.ibp:2}, we observe that, rearranging the sums, 
  $$
  \sum_{T\in\Th}\sum_{F\in\Fh[T]}
  ((\vel\SCAL\normal_{TF})\unw[F],\unv[F])_F
  = \sum_{F\in\Fhi}\sum_{T\in\Th[F]}  ((\vel\SCAL\normal_{TF})\unw[F],\unv[F])_F
  + \sum_{F\in\Fhb}((\vel\SCAL\normal_F)\unw[F],\unv[F])_F.
  $$
  Using, for all $F\in\Fhi$ such that $F\subset\partial T_1\cap\partial T_2$, the fact that $\vel\SCAL\normal_{T_1F}=-\vel\SCAL\normal_{T_2F}=\vel\SCAL\normal_F$, we infer that the first addend in the right-hand side is zero. 

\noindent\emph{Step 2}. Owing to~\eqref{def:Peclet} (see also \eqref{eq:def:Apmdeg} if $\diff[F]=0$) and since $\Ap(s)-\Am(s)=s$,
we observe that, for all $T\in\Th$ and all $F\in\Fh[T]$,
\begin{equation}\label{eq:Ap-Am=Id}
\tfrac{\diff[F]}{h_F}\Ap(\Pe_{TF})-\tfrac{\diff[F]}{h_F}\Am(\Pe_{TF})
=\vel\SCAL\normal_{TF}.
\end{equation}
Owing to~\eqref{eq:GBT.ibp}, we infer that, for all $\sw[h],\sv[h]\in\Uh$,
\begin{subequations}
  \label{eq:ah.vel.mu:formulas}
  \begin{align}
    \label{eq:ah.vel.mu:orig}
    &a_{\vel,\reac,h}(\sw[h],\sv[h])\nonumber
    \\
    &=\sum_{T\in\Th}\left\{
    -(\unw[T],\GBT\sv)_T  
    +(\reac\unw[T],\unv[T])_T
    \right\}
    +s_{\vel,h}\pneg(\sw[h],\sv[h])
    +\sum_{F\in\Fhb}(\tfrac{\diff[F]}{h_F}\Ap(\Pe_F)\unw[F],\unv[F])_F
    \\
    \label{eq:ah.vel.mu:bis}
    &\;=\sum_{T\in\Th}\left\{
    (\GBT\sw,\unv[T])_T
    +(\reac\unw[T],\unv[T])_T
    \right\}
    +s_{\vel,h}\ppos(\sw[h],\sv[h])
    +\sum_{F\in\Fhb}(\tfrac{\diff[F]}{h_F}\Am(\Pe_F)\unw[F],\unv[F])_F,
  \end{align}
\end{subequations}
where the global stabilization bilinear forms $s_{\vel,h}^\pm$ on $\Uh\times\Uh$ are assembled element-wise by setting $s_{\vel,h}^\pm(\sw[h],\sv[h])\eqbydef\sum_{T\in\Th} s_{\vel,T}^\pm(\sw,\sv)$.

\noindent\emph{Step 3}. Let $\sv[h]\in\Uh$.
  Using~\eqref{eq:GBT.ibp} with $\sw[h]=\sv[h]$ and \eqref{eq:Ap-Am=Id}, we infer that
  \begin{equation}
    \label{eq:stab:ah.vel.mu:1}
    \begin{aligned}
      -\sum_{T\in\Th}(\GBT\sv,\unv[T])_T
      &= \sum_{T\in\Th}\sum_{F\in\Fh[T]}
      \left(\tfrac{\frac{\diff[F]}{h_F}\Ap(\Pe_{TF})-\frac{\diff[F]}{h_F}\Am(\Pe_{TF})}{2}(\unv[F]-\unv[T]),\unv[F]-\unv[T]\right)_F
      \\
      &\quad -\sum_{F\in\Fhb}\left(\tfrac{\frac{\diff[F]}{h_F}\Ap(\Pe_F)-\frac{\diff[F]}{h_F}\Am(\Pe_F)}{2}\unv[F],\unv[F]\right)_F.
    \end{aligned}
  \end{equation}
  Taking $\sw[h]=\sv[h]$ in~\eqref{eq:ah.vel.mu:orig} and using~\eqref{eq:stab:ah.vel.mu:1} to substitute the first term in the right-hand side, we obtain
  $$
  \begin{aligned}
    a_{\vel,\reac,h}(\sv[h],\sv[h])
    &\ge\sum_{T\in\Th}\left\{
    \sum_{F\in\Fh[T]}\left(\tfrac{\frac{\diff[F]}{h_F}\Ap(\Pe_{TF})+\frac{\diff[F]}{h_F}\Am(\Pe_{TF})}{2}(\unv[F]-\unv[T]),\unv[F]-\unv[T]\right)_F
    +\reac_0\norm[T]{\unv[T]}^2
    \right\}
    \\
    &\quad+\sum_{F\in\Fhb}\left(\tfrac{\frac{\diff[F]}{h_F}\Ap(\Pe_F)+\frac{\diff[F]}{h_F}\Am(\Pe_F)}{2}\unv[F],\unv[F]\right)_F,
  \end{aligned}
$$
and the conclusion follows recalling~\eqref{eq:Lvel.tc.velc} and since
  $\Aabs(s)=\Ap(s)+\Am(s)$.
\end{proof}

\begin{proof}[Proof of~\eqref{eq:coer.ah}]
Sum the results of Lemmas~\ref{lem:stab.cont.ah.nu} and~\ref{lem:stab:ah.vel.mu}.
\end{proof}

\begin{proof}[Proof of the inf-sup condition~\eqref{eq:inf-sup.ah}]
  The proof hinges on the use of the locally scaled advective derivative as a test function, an idea which can be found, e.g., in the work of Johnson and Pitk\"{a}ranta~\cite{Johnson.Pitkaranta:86}.
  We denote by $\$$ the supremum in the right-hand side of~\eqref{eq:inf-sup.ah}.
  Let $\sw[h]\in\Uh$ and define $\sv[h]\in\Uh$ such that,
  \begin{equation}
    \label{eq:inf-sup.ah:svh}
    \unv[T]=h_T\velc^{-1}(\GBT\sw)\quad\forall T\in\Th,
    \qquad
    \unv[F]\equiv 0\quad \forall F\in\Fh.
  \end{equation}
  The following result is proved in Lemma~\ref{lem:bnd.vh.wh}:
  \begin{equation}
    \label{eq:inf-sup.ah:norm.svh}
    \norm[\sharp,h]{\sv[h]}\lesssim\norm[\sharp,h]{\sw[h]}.
  \end{equation}
  Using~\eqref{eq:inf-sup.ah:svh} in~\eqref{eq:ah} and recalling~\eqref{eq:ah.vel.mu:bis}, it is inferred that
  \begin{equation}
    \label{eq:inf-sup.ah:0}
    \begin{aligned}
      \sum_{T\in\Th} h_T\velc^{-1}\norm[T]{\GBT\sw}^2
      &= a_h(\sw[h],\sv[h]) 
      -a_{\diff,h}(\sw[h],\sv[h])       
      -\sum_{T\in\Th}(\reac\unw[T],\unv[T])_T
      \\
      &\quad
      -s_{\vel,h}\ppos(\sw[h],\sv[h])
      -\sum_{F\in\Fhb}(\tfrac{\diff[F]}{h_F}\Am(\Pe_F)\unw[F],\unv[F])_F.
    \end{aligned}
  \end{equation}
  Denote by $\term_1,\ldots,\term_5$ the addends in the right-hand side of~\eqref{eq:inf-sup.ah:0}.
  Using~\eqref{eq:inf-sup.ah:norm.svh}, we have 
  \begin{equation}
    \label{eq:inf-sup.ah:T1}
    |\term_1|\le\$\norm[\sharp,h]{\sv[h]}\lesssim\$\norm[\sharp,h]{\sw[h]}.
  \end{equation}
	Since $\unv[F]=0$ for any face $F$, using the Cauchy--Schwarz inequality on
the positive semi-definite bilinear form $a_{\diff,T}$ and recalling the definition
\eqref{eq:norm.nu} of $\norm[\diff,h]{{\cdot}}$, it is inferred, thanks to
\eqref{eq:inf-sup.ah:norm.svh}, that
  \begin{equation}
    \label{eq:inf-sup.ah:T2}
    |\term_2|
    \le\norm[\diff,h]{\sw[h]}\norm[\diff,h]{\sv[h]}
    \lesssim\norm[\flat,h]{\sw[h]}\norm[\sharp,h]{\sw[h]}.
  \end{equation}
  The estimate on $\term_3$ is trivial:
  \begin{equation}
    \label{eq:inf-sup.ah:T3}
    \begin{aligned}
      |\term_3|&\lesssim
      \norm[\vel,\reac,h]{\sw[h]}\norm[\vel,\reac,h]{\sv[h]}
      \lesssim\norm[\flat,h]{\sw[h]}\norm[\sharp,h]{\sw[h]}.
    \end{aligned}
  \end{equation}
  Let us now turn to $\term_4$. Using Remark \ref{rem:A2} (if $\diff[F]>0$)
and \ASSUM{A}{3} (otherwise) we see that
 	\begin{equation*}
	|\tfrac{\diff[F]}{h_F}\Apm(\Pe_{TF})|\lesssim \frac{\diff[F]}{h_F}+
	\frac{\diff[F]}{h_F}\Aabs(\Pe_{TF}).
	\label{eq:estApAabs}
	\end{equation*}
	Using the fact that $\diff[F]\le \diff[T]$ and $h_T\lesssim h_F$ owing to~\eqref{eq:hF}
whenever $F\in\Fh[T]$,
	the norm equivalence \eqref{eq:equiv:norms},
  the Cauchy--Schwarz inequality, and definition~\eqref{eq:norm.vel.mu} of the advective seminorm $\norm[\vel,\reac,h]{{\cdot}}$, we therefore find
    \begin{align}
      |\term_4|&\le
      \sum_{T\in\Th}\sum_{F\in\Fh[T]}(|\tfrac{\diff[F]}{h_F}\Ap(\Pe_{TF})|\,|\unw[F]-\unw[T]|,|\unv[F]-\unv[T]|)_F
			\nonumber\\
      &\lesssim \sum_{T\in\Th}\sum_{F\in\Fh[T]}\tfrac{\diff[T]}{h_T}(|\unw[F]-\unw[T]|,|\unv[F]-\unv[T]|)_F
      \nonumber\\
      &\qquad
      + \sum_{T\in\Th}\sum_{F\in\Fh[T]}(\tfrac{\diff[F]}{h_F}\Aabs(\Pe_{TF})\,|\unw[F]-\unw[T]|,|\unv[F]-\unv[T]|)_F
			\nonumber\\
      &\lesssim \norm[\diff,h]{\sw[h]}\norm[\diff,h]{\sv[h]}
      +  \norm[\vel,\reac,h]{\sw[h]}\norm[\vel,\reac,h]{\sv[h]}
      \lesssim\norm[\flat,h]{\sw[h]}\norm[\sharp,h]{\sv[h]}.
    \label{eq:inf-sup.ah:T4}
    \end{align}
  Proceeding similarly, it is inferred for $\term_5$ that
    \begin{equation}
    \label{eq:inf-sup.ah:T5}
    |\term_5|\lesssim \norm[\flat,h]{\sw[h]}\norm[\sharp,h]{\sw[h]}.
    \end{equation}
  Hence, using~\eqref{eq:inf-sup.ah:T1}--\eqref{eq:inf-sup.ah:T5} to bound the right-hand side of~\eqref{eq:inf-sup.ah:0}, we obtain
  \begin{equation}
    \label{eq:inf-sup.ah:1}
    \sum_{T\in\Th}h_T\velc^{-1}\norm[T]{\GBT\sw}^2
    \lesssim
    \$\norm[\sharp,h]{\sw[h]} +\norm[\flat,h]{\sw[h]}\norm[\sharp,h]{\sw[h]}.
  \end{equation}
  Adding $\norm[\flat,h]{\sw[h]}^2$ to both sides of inequality~\eqref{eq:inf-sup.ah:1}, and observing that, as a consequence of~\eqref{eq:coer.ah}, 
  \begin{equation}
    \label{eq:inf-sup.ah:2}
    \norm[\flat,h]{\sw[h]}^2 \le \xi^{-1}
    \frac{a_h(\sw[h],\sw[h])}{\norm[\sharp,h]{\sw[h]}}\norm[\sharp,h]{\sw[h]}
    \le \xi^{-1}\$\norm[\sharp,h]{\sw[h]},
  \end{equation}
  we infer the existence of $C$ depending on $\varrho$, $d$, and $k$, 
  but independent of $h$, $\diff$, $\vel$, and $\reac$, such that
  \[
  C\norm[\sharp,h]{\sw[h]}^2
  \le
  \xi^{-1}\$\norm[\sharp,h]{\sw[h]} + 
  \norm[\flat,h]{\sw[h]}\norm[\sharp,h]{\sw[h]}
  \le
  \xi^{-1}\$\norm[\sharp,h]{\sw[h]}
  + \frac{1}{2C}\norm[\flat,h]{\sw[h]}^2
  + \frac{C}{2}\norm[\sharp,h]{\sw[h]}^2,
  \]
  and the result follows using again~\eqref{eq:inf-sup.ah:2} for the second term in the right-hand side.
\end{proof}

\begin{lemma}
  \label{lem:bnd.vh.wh} Under the assumptions of Lemma \ref{lem:stab.ah},
  let $\sw[h]\in\Uh$ and $\sv[h]\in\Uh$ be defined as in~\eqref{eq:inf-sup.ah:svh}. 
  Then, \eqref{eq:inf-sup.ah:norm.svh} holds.
\end{lemma}

\begin{proof}
  Using~\eqref{eq:GBT.bis}, we observe that, for all $\sz\in\UT$, 
  \begin{equation}
    \label{eq:tech.res:1}
    \begin{aligned}
      \sqrt{\diff[T]}\norm[T]{\GBT\sz}
      &=\sup_{w\in\Poly{k}(T),\norm[T]{w}=1}\sqrt{\diff[T]}(\GBT\sz,w)_T
      \\
      &\lesssim\left\{\diff[T]
      \norm[T]{\vel\SCAL\GRAD\unz[T]}^2
      +\sum_{F\in\Fh[T]}\tfrac{\diff[T]}{h_F}\norm{|\vel\SCAL\normal_{TF}|(\unz[F]-\unz[T])}^2
      \right\}^{\nicefrac12}
      \lesssim\velc\norm[\diff,T]{\sz}.
    \end{aligned}
  \end{equation}
  The first inequality results from multiple applications of the Cauchy--Schwarz inequality together with the discrete trace inequality~\eqref{eq:trace.disc}
and the bound~\eqref{eq:Np} on $N_\partial$, while the second is an immediate consequence of definition
\eqref{eq:Lvel.tc.velc} of $\velc$ and of the equivalence \eqref{eq:equiv:norms}.
  \begin{asparaenum}[(i)]
  \item \emph{Diffusive contribution.}
    Recalling \eqref{eq:equiv:norms}, using the discrete inverse~\eqref{eq:inv} and trace~\eqref{eq:trace.disc} inequalities followed by~\eqref{eq:hF} to write $h_T/h_F\le\varrho^{-2}$ and the bound~\eqref{eq:Np} on $N_\partial$ for the boundary term, it is inferred that
    \begin{equation}
      \label{eq:tech.res:2}
      \begin{aligned}
        \norm[\diff,h]{\sv[h]}^2
        &\lesssim\sum_{T\in\Th}\left\{
        \diff[T]h_T^2\velc^{-2}\norm[T]{\GRAD\GBT\sw}^2
        +\sum_{F\in\Fh[T]}\diff[T]h_T\velc^{-2}\norm[F]{\GBT\sw}^2
        \right\}
        \\
        &\lesssim\sum_{T\in\Th}\diff[T]\velc^{-2}\norm[T]{\GBT\sw}^2
				\lesssim\norm[\diff,h]{\sw[h]}^2,
      \end{aligned}
    \end{equation}
    where, for all $T\in\Th$, we have used~\eqref{eq:tech.res:1} with $\sz=\sw$ to conclude.
  \item \emph{Advective and reactive contributions.}
		If $\diff[F]>0$ then, since $\Aabs$ is Lipschitz-continuous and vanishes at $0$,
\[
\tfrac{\diff[F]}{h_F}\Aabs(\Pe_{TF})\lesssim \tfrac{\diff[F]}{h_F}|\Pe_{TF}|=|\vel\SCAL\normal_{TF}|\le \velc.
\] 
Owing to \ASSUM{A}{3}, this inequality is also valid in the case $\diff[F]=0$. Hence, recalling definition~\eqref{eq:inf-sup.ah:svh} of $\sv[h]$ and using the discrete trace inequality~\eqref{eq:trace.disc},
    it is inferred, for all $T\in\Th$ and all $F\in\Fh[T]$,
    \begin{equation}
      \label{eq:tech.res:4}
      \norm[F]{\big[\tfrac{\diff[F]}{h_F}\Aabs(\Pe_{TF})\big]^{\nicefrac12}(\unv[F]-\unv[T])}
      = \norm[F]{\big[\tfrac{\diff[F]}{h_F}\Aabs(\Pe_{TF})\big]^{\nicefrac12}\unv[T]}
      \lesssim h_T^{\nicefrac12}\velc^{-\nicefrac12}\norm[T]{\GBT\sw}.
    \end{equation}
    Using~\eqref{eq:tech.res:4} together with the uniform bound~\eqref{eq:Np} on $N_\partial$
and the definition~\eqref{eq:inf-sup.ah:svh} of $\sv[h]$, we deduce that
    \begin{align}
      \norm[\vel,\mu,h]{\sv[h]}^2
      \lesssim&\sum_{T\in\Th}\left\{         
      h_T\velc^{-1}\norm[T]{\GBT\sw}^2
      + h_T^2 \tc^{-1}\velc^{-2}\norm[T]{\GBT\sw}^2
      \right\}
			\nonumber\\
      \lesssim&\sum_{T\in\Th} h_T\velc^{-1}\norm[T]{\GBT\sw}^2,
      \label{eq:tech.res:5}
    \end{align}
    where the conclusion follows by noticing that \eqref{eq:err.est.ass} yields $h_T\velc^{-1}\tc^{-1}\le 1$.
    Moreover, recalling \eqref{eq:GBT} and using the Cauchy--Schwarz and inverse~\eqref{eq:inv} inequalities together with definition~\eqref{eq:Lvel.tc.velc} of $\velc$ to infer $|(\unv[T],\vel\SCAL\GRAD w)_T|\le \norm[T]{\unv[T]}\velc C_{\rm inv} h_T^{-1}\norm[T]{w}$, one has, for all $T\in\Th$,
    \begin{equation}
      \label{eq:tech.res:6}
      \norm[T]{\GBT\sv}
        =\sup_{w\in\Poly{k}(T),\norm[T]{w}=1} -(\unv[T],\vel\SCAL\GRAD w)_T%
      \lesssim\velc h_T^{-1}\norm[T]{\unv[T]}
      =\norm[T]{\GBT\sw},
    \end{equation}
    where we have used the definition~\eqref{eq:inf-sup.ah:svh} of $\unv[T]$ to conclude.
    Hence, using~\eqref{eq:tech.res:5} and~\eqref{eq:tech.res:6}, we estimate the advective and reactive contributions to $\norm[\sharp,h]{\sv[h]}$ as follows:
    \begin{equation}
      \label{eq:tech.res:7}
      \norm[\vel,\reac,h]{\sv[h]}^2
      + \sum_{T\in\Th}h_T\velc^{-1}\norm[T]{\GBT\sv}^2
      \lesssim
      \sum_{T\in\Th}h_T\velc^{-1}\norm[T]{\GBT\sw}^2.
    \end{equation}
    The conclusion then follows from~\eqref{eq:norm.h} recalling~\eqref{eq:tech.res:2} and~\eqref{eq:tech.res:7}.
  \end{asparaenum}
\end{proof}

\subsection{Error analysis}
\label{sec:proof_err}

Here we prove Theorem \ref{thm:err.est}. Owing to~\eqref{eq:inf-sup.ah}, we infer that
  \begin{equation}
    \label{eq:abs.est}
    \norm[\sharp,h]{\shu[h]-\su[h]}
    \leq (\gamma \xi)^{-1}
\sup_{\sv[h]\in\Uh\setminus\{0\}}\frac{{\cal E}_h(\sv[h])}{\norm[\sharp,h]{\sv[h]}},
  \end{equation}
  where 
  $$
  {\cal E}_h(\sv[h])\eqbydef 
  a_h(\shu[h]-\su[h],\sv[h])
  =a_h(\shu[h],\sv[h])-l_h(\sv[h])
  =a_{\diff,h}(\shu[h],\sv[h]) + a_{\vel,\reac,h}(\shu[h],\sv[h]) - l_h(\sv[h])
  $$ 
  is the consistency error.   
  We derive a bound for this quantity for a generic $\sv[h]\in\Uh$ proceeding in the same spirit as~\cite[Theorem~8]{Di-Pietro.Ern.ea:14}.
  Recalling that $f=\DIV(-\diff\GRAD u + \vel u)+\mu u$ a.e. in $\Omega$, we perform an element-by-element integration by parts
on the first term in the definition~\eqref{eq:lh} of $l_h(\sv[h])$. We then use the conservation property 
$$
\restrto{(-\diff\GRAD u + \vel u)}{T_1}\SCAL\normal_{T_1F}
+\restrto{(-\diff\GRAD u + \vel u)}{T_2}\SCAL\normal_{T_2F}=0,
$$ 
which is valid for any interface $F\subset\partial T_1\cap\partial T_2$, to introduce $\unv[F]$ in the resulting sums.
We also notice that, for any face $F\in\Fhb$, $\frac{\diff[F]}{h_F}\Am(\Pe_F)g=\frac{\diff[F]}{h_F}\Am(\Pe_F)u$ on $F$,
which results from the boundary condition \eqref{eq:strong:BC} if $\diff[F]>0$ and from definition
\eqref{eq:def:Apmdeg} if $\diff[F]=0$. Letting $\cu\eqbydef\pT\shu[T]$ and
using definitions \eqref{eq:aT.nu} and \eqref{eq:ah.nu}
for $a_{\diff,h}$, and \eqref{eq:ah.vel.mu:orig} and \eqref{eq:GBT.bis} for $a_{\vel,\reac,h}$,
we then find
  \begin{equation}
    \label{eq:right}
    \begin{aligned}
      &\mathcal{E}_h(\sv[h])=\\
      &
      \sum_{T\in\Th}\left\{
      (\diff[T]\GRAD (\cu-u),\GRAD\unv[T])_T 
      + \sum_{F\in\Fh[T]}(\diff[T]\GRAD (\cu- \restrto{u}{T})\SCAL\normal_{TF},\unv[F]-\unv[T])_F
       + s_{\diff,T}(\shu[h],\sv[h])\right\}
      \\
      &
      {+} \sum_{T\in\Th} \left\{
      (u-\unhu[T], \vel\SCAL\GRAD\unv[T]+\mu\unv[T])_T + \sum_{F\in\Fh[T]}((\vel\SCAL\normal_{TF})(\restrto{u}{T}-\unhu[T]),\unv[F]-\unv[T])_F
      {+} s_{\vel,T}^-(\shu[h],\sv[h])\right\}
      \\
      & + \sum_{F\in\Fhb}\left\{
      (\diff[F](\GRAD (u - \cu)\SCAL\normal_{TF}, \unv[F])_F
      +\tfrac{\diff[F]}{h_F}(\Ap(\Pe_{F})(\unhu[F]-u),\unv[F])_F
      \right\}.
    \end{aligned}
  \end{equation}
  We have used the fact that $\sum_{F\in\Fhb}\frac{\varsigma\diff[F]}{h_F}(\unhu[F]-g,\unv[F])_F=0$.
Indeed, for all $F\in\Fhb$,
either $\diff[F]=0$ and the corresponding addend vanishes, or $\diff[F]>0$ so that $F\subset\Gam$ (cf.~\eqref{eq:Gamma}) and hence $\unhu[F]=\lproj[F]{k}g$ owing to~\eqref{eq:strong:BC} and $(\unhu[F]-g,\unv[F])_F=(\lproj[F]{k}g-g,\unv[F])_F=0$ since $\unv[F]\in\Poly[d-1]{k}(F)$.

  Denote by $\term_1$, $\term_2$, $\term_3$ the lines composing the right-hand side of~\eqref{eq:right} and corresponding, respectively, to diffusive terms, advective terms, and weakly enforced boundary conditions.\\
  (i) \emph{Diffusive terms.}
  Proceeding as in the proof of~\cite[Theorem~8]{Di-Pietro.Ern.ea:14} yields 
  \begin{equation}
    \label{eq:err.est:T1}
    |\term_1|
    \lesssim 
    \left\{\sum_{T\in\Th}
    \diff[T]h_T^{2(k+1)}\norm[H^{k+2}(T)]{u}^2
    \right\}^{\nicefrac12}\norm[\diff,h]{\sv[h]}.
  \end{equation}
  Observe that, to obtain~\eqref{eq:err.est:T1}, a crucial point is the choice of interpolating $\unhu[F]$ from the diffusive side whenever $F\subset\Int[-]$ since this guarantees that $\cu$ enjoys the approximation properties~\eqref{eq:rT.approx} whenever $\diff[T]\neq 0$.
  \\
  (ii) \emph{Advective-reactive terms.}
  Denote by $\term_{2,1}$, $\term_{2,2}$, and $\term_{2,3}$ the three addends that compose $\term_2$.
  For the first term, observing that $(\lproj[T]{0}\vel)\SCAL\GRAD\unv[T]\in\Poly{k-1}(T)\subset\Poly{k}(T)$ and recalling that, owing to~\eqref{eq:Ih}, $\unhu[T]=\lproj[T]{k}u$, we infer that
  $
  \term_{2,1}= \sum_{T\in\Th}(u-\lproj[T]{k}u,(\vel-\lproj[T]{0}\vel)\SCAL\GRAD\unv[T]+\reac\unv[T])_T.
  $
  Hence, 
  \begin{equation}
  \begin{aligned}
    \label{eq:err.est:T21}
    |\term_{2,1}|
    &\lesssim
    \sum_{T\in\Th}\left\{
    \norm[L^{\infty}(T)^d]{\vel-\lproj[T]{0}\vel}
    \norm[T]{u-\lproj[T]{k}u}
    \norm[T]{\GRAD\unv[T]} + \norm[L^\infty(T)]{\mu}\norm[T]{u-\lproj[T]{k}u}\norm[T]{\unv[T]}\right\}
    \\
    &\lesssim
    \left\{\sum_{T\in\Th}
    \tc^{-1} h_T^{2(k+1)}\norm[H^{k+1}(T)]{u}^2
    \right\}^{\nicefrac12}\norm[\vel,\reac,h]{\sv[h]},
  \end{aligned}
  \end{equation}
  where the second inequality is obtained using the fact that $\vel$ is Lipschitz continuous to infer $\norm[L^\infty(T)^d]{\vel-\lproj[T]{0}\vel}\le \Lvel h_T$ followed by the inverse inequality~\eqref{eq:inv} together with the definition~\eqref{eq:Lvel.tc.velc} of $\tc$.

  To treat $\term_{2,2}$ and $\term_{2,3}$, we proceed differently according to the value of the local P\'eclet number. 
  We write $\term_{2,2}=\term_{2,2}^{\rm d}+\term_{2,2}^{\rm a}$ and $\term_{2,3}=\term_{2,3}^{\rm d}+\term_{2,3}^{\rm a}$, where the superscript ``d'' corresponds to integrals where $|\Pe_{TF}|\le 1$, while the superscript ``a'' corresponds to integrals where $|\Pe_{TF}|>1$ (which conventionally include all faces where $\diff[F]=0$).
  We denote by $\cf_{|\Pe_{TF}|\le 1}$ and $\cf_{|\Pe_{TF}|> 1}$ the two characteristic functions of these regions. The idea is that we use the diffusive norm of $\sv[h]$ if $|\Pe_{TF}|\le 1$, whereas we use the advective norm if $|\Pe_{TF}|>1$. Before proceeding, we observe that, for all $T\in\Th$ and all $F\in\Fh$, the following holds:
    \begin{equation}
      \label{eq:est.bnd.term}
      \norm[F]{\unhu[F]-\unhu[T]}
      =\norm[F]{\lproj[F]{k}(\restrto{u}{\Tdiff}-\unhu[T])}
      \le\norm[F]{\restrto{u}{\Tdiff}-\unhu[T]},
    \end{equation}
    where we have used that $\unhu[F]=\lproj[F]{k}\restrto{u}{\Tdiff}$ (see~\eqref{eq:Ih}) $\restrto{\unhu[T]}{F}\in\Poly[d-1]{k}(F)$, and that $\lproj[F]{k}$ is a projector. For $\term_{2,2}^{\rm d}$, it is also useful to notice that, since $\Am(0)=0$ and $\Am$ is Lipschitz-continuous,
  \begin{equation}\label{esterr:Am}
    \left|\tfrac{\diff[F]}{h_F}\Am(\Pe_{TF})\right|\lesssim\tfrac{\diff[F]}{h_F}|\Pe_{TF}|
    =|\vel\SCAL\normal_{TF}|\le\velc
  \end{equation}
  whenever $\diff[F]>0$ (which is always the case if $\cf_{|\Pe_{TF}|\le 1}\not\equiv 0$).
  Hence, observing that $\diff[F]>0$ indicates that the exact solution $u$ does not jump across $F$, so that we can simply write $\restrto{u}{T}$ in place of $\restrto{u}{\Tdiff}$,
    \begin{equation}
      \label{est:T22T23d}
      \begin{aligned}
       &|\term_{2,2}^{\rm d}|+|\term_{2,3}^{\rm d}| \\
        &\quad\lesssim \sum_{T\in\Th} \sum_{F\in\Fh[T]}(|\vel\SCAL\normal_{TF}|\cf_{|\Pe_{TF}|\le 1}\,|\restrto{u}{T}-\unhu[T]|,
        |\unv[F]-\unv[T]|)_F
        \\
        &\quad\qquad+\sum_{T\in\Th} \sum_{F\in\Fh[T]}\tfrac{\diff[F]}{h_F}(|\Pe_{TF}|\cf_{|\Pe_{TF}|\le 1}\,|\unhu[F]-\unhu[T]|,
        |\unv[F]-\unv[T]|)_F
        \\
        &\quad
        \lesssim \left\{ \sum_{T\in\Th} \sum_{F\in\Fh[T]} \tfrac{\diff[F]}{h_F}\norm[L^\infty(F)]{\Pe_{TF}\cf_{|\Pe_{TF}|\le 1}}^2
        \norm[F]{\restrto{u}{T}-\unhu[T]}^2\right\}^{\nicefrac12}
        \times\left\{ \sum_{T\in\Th} \sum_{F\in\Fh[T]} \tfrac{\diff[F]}{h_F}
        \norm[F]{\unv[F]-\unv[T]}^2\right\}^{\nicefrac12}
        \\
        &\quad
        \lesssim\left\{ \sum_{T\in\Th} \sum_{F\in\Fh[T]} \velc \min(1,\Pe_T)
        \norm[F]{\restrto{u}{T}-\unhu[T]}^2\right\}^{\nicefrac12}
        \norm[\diff,h]{\sv[h]},
      \end{aligned}
    \end{equation}%
where we have used~\eqref{eq:est.bnd.term} and~\eqref{esterr:Am} to bound the second addend in the first line and the norm equivalence \eqref{eq:equiv:norms} to conclude.
To estimate $\term_{2,2}^{\rm a}$, it suffices to observe that
\begin{equation}
  \begin{aligned}
    \label{est:T22a}
    |\term_{2,2}^{\rm a}|
    &\lesssim \left\{ \sum_{T\in\Th} \sum_{F\in\Fh[T]}
    \norm[L^\infty(F)]{\vel\SCAL\normal_{TF}\cf_{|\Pe_{TF}|>1}}
    \norm[F]{\restrto{u}{T}-\unhu[T]}^2\right\}^{\nicefrac12}
    \\
    &\qquad
    \times\left\{ \sum_{T\in\Th} \sum_{F\in\Fh[T]} 
    \norm[F]{|\vel\SCAL\normal_{TF}\cf_{|\Pe_{TF}|>1}|^{\nicefrac12}(\unv[F]-\unv[T])}^2
    \right\}^{\nicefrac12}
    \\
    &\lesssim\left\{ \sum_{T\in\Th} \sum_{F\in\Fh[T]} \velc \min(1,\Pe_T)
    \norm[F]{\restrto{u}{T}-\unhu[T]}^2\right\}^{\nicefrac12}
    \norm[\vel,\reac,h]{\sv[h]},
  \end{aligned}
\end{equation}%
where the introduction of the advective norm in the last inequality is justified since, owing to \eqref{eq:Ap-Am=Id} (see also \eqref{eq:def:Apmdeg} if $\diff[F]=0$)
and Assumption \ASSUM{A}{2}, 
\begin{equation}\label{eq:estvel:Aabs}
|\vel\SCAL\normal_{TF}|\cf_{|\Pe_{TF}|>1}\lesssim\tfrac{\diff[F]}{h_F}\Aabs(\Pe_{TF}).
\end{equation}
To estimate $\term_{2,3}^{\rm a}$, recalling~\eqref{eq:est.bnd.term} we observe that
$$
|\term_{2,3}^{\rm a}|\le
\sum_{T\in\Th}\sum_{F\in\Fh[T]}
(|\tfrac{\diff[F]}{h_F}\Am(\Pe_{TF})|\cf_{|\Pe_{TF}|>1} |\lproj[F]{k}(\restrto{u}{\Tdiff}-\unhu[T])|, |\unv[F]-\unv[T]|)_F.
$$
For given $T\in\Th$ and $F\in\Fh[T]$, we have the following mutually exclusive cases:
\begin{inparaenum}[(i)]
\item $\diff[F]>0$ or ($\diff[F]=0$ and $F\subset \Int[+]$), in which case
$\restrto{u}{\Tdiff}=\restrto{u}{T}$ since $u$ does not have a jump at $F$
(see \eqref{eq:match.u} if $F\subset \Int[+]$);
\item $\diff[F]=0$ and $F\subset \Int[-]$, in which case, 
recalling \eqref{eq:def:Apmdeg}, $\frac{\diff[F]}{h_F}\Am(\Pe_{TF})=
(\vel\SCAL\normal_{TF})^-=0$.
\end{inparaenum}
Hence, in any case, $|\frac{\diff[F]}{h_F}\Am(\Pe_{TF})| |\lproj[F]{k}(\restrto{u}{\Tdiff}-\unhu[T])|
=|\frac{\diff[F]}{h_F}\Am(\Pe_{TF})| |\lproj[F]{k}(\restrto{u}{T}-\unhu[T])|$.
Using this fact and observing that, for all $T\in\Th$ and all $F\in\Fh[T]$, $|\frac{\diff[F]}{h_F}\Am(\Pe_{TF})|\lesssim|\vel\SCAL\normal_{TF}|\lesssim\velc$, we infer the estimate
\begin{equation}
  \label{eq:T23a}
  |\term_{2,3}^{\rm a}|\lesssim
  \left\{ \sum_{T\in\Th} \sum_{F\in\Fh[T]} \velc \min(1,\Pe_T)
  \norm[F]{\restrto{u}{T}-\unhu[T]}^2\right\}^{\nicefrac12}\norm[\vel,\reac,h]{\sv[h]}.
\end{equation}
To conclude the estimate on $\term_{2,2}$ and $\term_{2,3}$, we collect the bounds \eqref{est:T22T23d},~\eqref{est:T22a}, and~\eqref{eq:T23a}, and invoke \eqref{eq:approx.lproj.T} to write $\norm[F]{\restrto{u}{T}-\unhu[T]} \le C_{\rm app} h_T^{k+\nicefrac12} \seminorm[H^{k+1}(T)]{u}$, so that
\begin{equation}
  \label{eq:err.est:T22T23}
  |\term_{2,2}|+|\term_{2,3}|\lesssim \left\{
  \sum_{T\in\Th} \velc  \min(1,\Pe_{T}) h_T^{2k+1} \norm[H^{k+1}(T)]{u}^2\right\}^{\nicefrac12}\norm[\sharp,h]{\sv[h]}.
\end{equation}
\\
(iii) \emph{Weakly enforced boundary conditions.}
Let us now estimate $\term_3$. Denoting by $\term_{3,1}$ and $\term_{3,2}$ the two
addends in $\term_3$, the estimate of $\term_{3,1}$ is a straightforward consequence
of the Cauchy-Schwarz inequality, the definition \eqref{eq:norm.nu} of $\norm[\diff,h]{{\cdot}}$,
and the approximation property \eqref{eq:rT.approx} of $\cu=\pT\shu$:
\begin{equation}
    \label{est:T31}
    |\term_{3,1}|\le \left\{
    \sum_{F\in\Fhb} \diff[F]h_F\norm[F]{\GRAD(u-\cu[T(F)])}^2
    \right\}^{\nicefrac12}
    \norm[\diff,h]{\sv[h]}
    \le \left\{\sum_{F\in\Fhb} \diff[F]h_{T(F)}^{2(k+2)}\norm[H^{k+2}(T(F))]{u}^2\right\}^{\nicefrac12}
    \norm[\sharp,h]{\sv[h]}.
\end{equation}
To estimate $\term_{3,2}$, we apply ideas similar to those employed for bounding $\term_{2,2}$.
We first observe that, for all $F\in\Fhb$,
\begin{equation}\label{est:T22T23u2}
  \norm[F]{u-\unhu[F]} \lesssim h_T^{k+1/2} \seminorm[H^{k+1}(T)]{u}.
\end{equation}%
Since $|\frac{\diff[F]}{h_F}\Ap(\Pe_{TF})|\lesssim
|\vel\SCAL\normal_{TF}|$ (proved as for $\Am$ above) and
$|\Ap(\Pe_{TF})|\lesssim |\Pe_{TF}|$ whenever $\diff[F]>0$,
invoking the definitions \eqref{eq:norm.nu} and \eqref{eq:norm.vel.mu} of the diffusive and
advective norms and reasoning as in the estimates of $\term_{2,2}^{\rm d}$
and $\term_{2,2}^{\rm a}$, estimate \eqref{eq:estvel:Aabs} and the approximation property 
\eqref{est:T22T23u2} yield
\begin{equation}
  \label{est:T32}
  \begin{aligned}    
    |\term_{3,2}|\lesssim&\sum_{F\in\Fhb}
    \tfrac{\diff[F]}{h_F}(|\Pe_{F}|\cf_{|\Pe_{F}|\le 1}\,|\unhu[F]-u|,|\unv[F]|)_F+
    \sum_{F\in\Fhb}(|\vel\SCAL\normal_{TF}|\cf_{|\Pe_{F}|> 1}\,|\unhu[F]-u|,|\unv[F]|)_F
    \\
    \lesssim&\left\{\sum_{F\in\Fhb}
    \velc \norm[L^\infty(F)]{\Pe_{F}\cf_{|\Pe_{F}|\le 1}} h_T^{2k+1}\norm[H^{k+1}(T)]{u}^2
    \right\}^{\nicefrac12}
    \norm[\diff,h]{\sv[h]}
    \\
    &+\left\{
    \sum_{F\in\Fhb}
    \norm[L^\infty(F)]{\vel\SCAL\normal_{TF}\cf_{|\Pe_{F}|> 1}}h_T^{2k+1}\norm[H^{k+1}(T)]{u}^2
    \right\}^{\nicefrac12}
    \norm[\vel,\reac,h]{\sv[h]}
    \\
    \lesssim&\left\{
    \sum_{F\in\Fhb,\,F\subset\partial T}
    \velc \min(1,\Pe_T)h_T^{2k+1}\norm[H^{k+1}(T)]{u}^2
    \right\}^{\nicefrac12}
    \norm[\sharp,h]{\sv[h]}.
  \end{aligned}
\end{equation}
The proof is completed by plugging estimates \eqref{eq:err.est:T1}, \eqref{eq:err.est:T21},
\eqref{eq:err.est:T22T23}, \eqref{est:T31}, and \eqref{est:T32}
into \eqref{eq:right}, and using the resulting bound to estimate the right-hand side of~\eqref{eq:abs.est}.


\footnotesize
\bibliographystyle{plain}
\bibliography{adho}

\begin{thebibliography}{10}

\bibitem{Bassi.Rebay.ea:97}
F.~Bassi, S.~Rebay, G.~Mariotti, S.~Pedinotti, and M.~Savini.
\newblock A high-order accurate discontinuous finite element method for
  inviscid and viscous turbomachinery flows.
\newblock In R.~Decuypere and G.~Dibelius, editors, {\em Proceedings of the
  $2^{\mathrm{nd}}$ European Conference on Turbomachinery Fluid Dynamics and
  Thermodynamics}, pages 99--109, 1997.

\bibitem{Beirao-da-Veiga.Droniou.ea:10}
L.~Beir\~{a}o~da Veiga, J.~Droniou, and M.~Manzini.
\newblock A unified approach for handling convection terms in finite volumes
  and mimetic discretization methods for elliptic problems.
\newblock {\em IMA J. Numer. Anal.}, 31(4):1357--1401, 2010.

\bibitem{Brezzi.Lipnikov.ea:05}
F.~Brezzi, K.~Lipnikov, and M.~Shashkov.
\newblock Convergence of the mimetic finite difference method for diffusion
  problems on polyhedral meshes.
\newblock {\em SIAM J. Numer. Anal.}, 43(5):1872--1896, 2005.

\bibitem{Chainais.Droniou:11}
C.~Chainais-Hillairet and J.~Droniou.
\newblock Finite-volume schemes for noncoercive elliptic problems with
  {N}eumann boundary conditions.
\newblock {\em IMA J. Numer. Anal.}, 31(1):61--85, 2011.

\bibitem{Chen.Cockburn:14}
Y.~Chen and B.~Cockburn.
\newblock Analysis of variable-degree {HDG} methods for convection-diffusion
  equations. {P}art {II}: {S}emimatching nonconforming meshes.
\newblock {\em Math. Comp.}, 83(285):87--111, 2014.

\bibitem{Cockburn.Dong.ea:09}
B.~Cockburn, B.~Dong, J.~Guzm\'{a}n, M.~Restelli, and R.~Sacco.
\newblock A hybridizable discontinuous {Galerkin} method for steady-state
  convection-diffusion-reaction problems.
\newblock {\em SIAM J. Sci. Comput.}, 31(5):3827--3846, 2009.

\bibitem{Di-Pietro.Ern:12}
D.~A. Di~Pietro and A.~Ern.
\newblock {\em Mathematical {A}spects of {D}iscontinuous {G}alerkin {M}ethods},
  volume~69 of {\em Math. Appl.}
\newblock Springer-Verlag, Berlin, 2012.

\bibitem{Di-Pietro.Ern:14a}
D.~A. Di~Pietro and A.~Ern.
\newblock A hybrid high-order locking-free method for linear elasticity on
  general meshes.
\newblock {\em Comput. Methods Appl. Mech. Engrg.}, 283:1--21, 2015.

\bibitem{Di-Pietro.Ern:14b}
D.~A. Di~Pietro and A.~Ern.
\newblock Hybrid high-order methods for variable-diffusion problems on general
  meshes.
\newblock {\em C. R. Math. Acad. Sci Paris}, 353:31--34, 2015.

\bibitem{Di-Pietro.Ern.ea:08}
D.~A. Di~Pietro, A.~Ern, and J.-L. Guermond.
\newblock Discontinuous {G}alerkin methods for anisotropic semidefinite
  diffusion with advection.
\newblock {\em SIAM J. Numer. Anal.}, 46(2):805--831, 2008.

\bibitem{Di-Pietro.Ern.ea:14}
D.~A. Di~Pietro, A.~Ern, and S.~Lemaire.
\newblock An arbitrary-order and compact-stencil discretization of diffusion on
  general meshes based on local reconstruction operators.
\newblock {\em Comput. Methods Appl. Math.}, 14(4):461--472, 2014.

\bibitem{Di-Pietro.Lemaire:14}
D.~A. Di~Pietro and S.~Lemaire.
\newblock An extension of the {Crouzeix--Raviart} space to general meshes with
  application to quasi-incompressible linear elasticity and {Stokes} flow.
\newblock {\em Math. Comp.}, 84(291):1--31, 2015.

\bibitem{Droniou:02}
J.~Droniou.
\newblock Non-coercive linear elliptic problems.
\newblock {\em Potential Anal.}, 17(2):181--203, 2002.

\bibitem{Droniou:10}
J.~Droniou.
\newblock Remarks on discretizations of convection terms in hybrid mimetic
  mixed methods.
\newblock {\em Netw. Heterog. Media}, 5(3):545--563, 2010.

\bibitem{Droniou.Eymard:06}
J.~Droniou and R.~Eymard.
\newblock A mixed finite volume scheme for anisotropic diffusion problems on
  any grid.
\newblock {\em Numer. Math.}, 105:35--71, 2006.

\bibitem{Droniou.Eymard.ea:10}
J.~Droniou, R.~Eymard, T.~Gallou\"{e}t, and R.~Herbin.
\newblock A unified approach to mimetic finite difference, hybrid finite volume
  and mixed finite volume methods.
\newblock {\em Math. Models Methods Appl. Sci.}, 20(2):265--295, 2010.

\bibitem{Dupont.Scott:80}
T.~Dupont and R.~Scott.
\newblock Polynomial approximation of functions in {S}obolev spaces.
\newblock {\em Math. Comp.}, 34(150):441--463, 1980.

\bibitem{Ern.Proft:06}
A.~Ern and J.~Proft.
\newblock Multi-algorithmic methods for coupled hyperbolic-parabolic problems.
\newblock {\em Int. J. Numer. Anal. Model.}, 3(1):94--114, 2006.

\bibitem{Eymard.Gallouet.ea:10}
R.~Eymard, T.~Gallou{\"e}t, and R.~Herbin.
\newblock Discretization of heterogeneous and anisotropic diffusion problems on
  general nonconforming meshes. {SUSHI}: {A} scheme using stabilization and
  hybrid interfaces.
\newblock {\em IMA J. Numer. Anal.}, 30(4):1009--1043, 2010.

\bibitem{Gastaldi.Quarteroni:89}
F.~Gastaldi and A.~Quarteroni.
\newblock On the coupling of hyperbolic and parabolic systems: {A}nalytical and
  numerical approach.
\newblock {\em Appl. Numer. Math.}, 6:3--31, 1989/90.

\bibitem{Herbin.Hubert:08}
R.~Herbin and F.~Hubert.
\newblock Benchmark on discretization schemes for anisotropic diffusion
  problems on general grids.
\newblock In R.~Eymard and J.-M. H\'{e}rard, editors, {\em Finite Volumes for
  Complex Applications V}, pages 659--692. John Wiley \& Sons, New York, 2008.

\bibitem{Houston.Schwab.ea:02}
P.~Houston, C.~Schwab, and E.~S\"uli.
\newblock Discontinuous $hp$-finite element methods for
  advection-diffusion-reaction problems.
\newblock {\em SIAM J. Numer. Anal.}, 39(6):2133--2163, 2002.

\bibitem{Johnson.Pitkaranta:86}
C.~Johnson and J.~Pitk{\"a}ranta.
\newblock An analysis of the discontinuous {G}alerkin method for a scalar
  hyperbolic equation.
\newblock {\em Math. Comp.}, 46(173):1--26, 1986.

\bibitem{Wang.Ye:13}
J.~Wang and X.~Ye.
\newblock A weak {Galerkin} element method for second-order elliptic problems.
\newblock {\em J. Comput. Appl. Math.}, 241:103--115, 2013.

\end{thebibliography}

\end{document}